\documentclass[10pt]{amsart}
\usepackage{eurosym}
\usepackage{amsfonts}
\usepackage{amsmath,amssymb,amsthm,color}
\usepackage{amsopn}
\usepackage{latexsym}
\usepackage{float}
\usepackage{bm}
\usepackage[utf8]{inputenc}
\usepackage[colorinlistoftodos]{todonotes}

\usepackage{hyperref}

\AtBeginDocument{
	\label{CorrectFirstPageLabel}
	
}

\renewcommand{\d}{\mathrm{d}}

\newcommand{\mg}{\mathfrak g }

\newcommand{\mq}{\mathfrak q }

\newcommand{\mn}{\mathfrak n }
\newcommand{\mr}{\mathfrak r }

\newcommand{\mz}{\mathfrak z }

\newcommand{\mc}{\mathfrak c }

\newcommand{\mh}{\mathfrak h }
\newcommand{\ma}{\mathfrak a }

\newcommand{\so}{\mathfrak{so} }
\renewcommand{\sp}{\mathfrak{sp} }

\newcommand{\su}{\mathfrak{su} }
\newcommand{\lela}{ g\left(}
\newcommand{\rira}{\right)}

\newcommand{\bs}{\backslash}

\newcommand{\SO}{{\rm SO}}
\newcommand{\SU}{{\rm SU}}
\newcommand{\U}{{\rm U}}
\newcommand{\Sp}{\rm Sp}

\newcommand{\R}{\mathbb R}
  \renewcommand{\H}{\mathbb H}

\newcommand{\C}{\mathbb C}

\newcommand{\vol}{\rm vol}
\newcommand{\rG}{\rm G}

\newcommand{\non}{\nonumber}

\DeclareMathOperator{\ad}{ad}

\DeclareMathOperator{\tr}{tr}
\DeclareMathOperator{\im}{im}
\DeclareMathOperator{\diag}{diag}

\numberwithin{equation}{section}

 \newtheorem{teo}{Theorem}[section]
 \newtheorem{pro}[teo]{Proposition}
 \newtheorem{cor}[teo]{Corollary}
 \newtheorem{lm}[teo]{Lemma}

 \theoremstyle{definition}
 \newtheorem{ex}[teo]{Example}
 \newtheorem{remark}[teo]{Remark}  
 \newtheorem{remarks}[teo]{Remarks}
 \newtheorem{assumption}[teo]{Assumption}

\newcommand{\nc}{\newcommand}
\nc{\Iso}{\operatorname{Iso}}
 \nc{\iso}{\mathfrak{iso}}
 \nc{\sso}{\mathfrak{so}}
\nc{\Ad}{\operatorname{Ad}} 
\nc{\Sym}{\mathrm{Sym}}
\nc{\Hol}{\mathrm{Hol}}

  \nc{\pr}{\operatorname{pr}} 
 \nc{\Dera}{\operatorname{Dera}} \nc{\Auto}{\operatorname{Auto}}
 \nc{\LL}{{\rm L}}
\nc{\dd}{{\rm d}}
\nc{\Id}{{\rm Id}}

\newcommand{\qandq}{\quad\text{and}\quad}

\newcommand{\qforq}{\quad\text{for}\quad}

\newcommand{\be}{{\bf e}}


\author[A. Clarke]{Andrew Clarke}
\address{A.~Clarke: Instituto de Matem\'atica, Universidade Federal do Rio de Janeiro, Av. Athos da Silveira Ramos 149, Rio de Janeiro, RJ, 21941-909, Brazil}
\email{andrew@im.ufrj.br} 

\author[V. del Barco]{Viviana del Barco}
\address{V.~del Barco: Instituto de Matemática, Estatística e Computação Científica, Universidade Estadual de Campinas,  Rua Sergio Buarque de Holanda, 651, Cidade Universitaria Zeferino Vaz, 13083-859, Campinas, São Paulo, Brazil.}
\email{delbarc@ime.unicamp.br}

\author[A.J. Moreno]{Andrés J. Moreno}
\address{A.~J.~Moreno: Instituto de Matemática, Estatística e Computação Científica, Universidade Estadual de Campinas,  Rua Sergio Buarque de Holanda, 651, Cidade Universitaria Zeferino Vaz, 13083-859, Campinas, São Paulo, Brazil.}
\email{amoreno@unicamp.br}

 \makeindex

\begin{document}

\title{$\rG_2$-instantons on $2$-step nilpotent Lie groups}

\date{\today}

\begin{abstract}
    We study the $\rG_2$-instanton condition for a family of metric connections arisen from the characteristic connection, on $7$-dimensional $2$-step nilpotent Lie groups with left-invariant coclosed $\rG_2$-structures. According to the dimension of the commutator subgroup, we establish necessary and sufficient conditions for the connection to be an instanton, in terms of the torsion of the $\rG_2$-structure, the torsion of the connection and the Lie group structure. Moreover, we show that in our setup, $\rG_2$-instantons define a naturally reductive structure on the simply connected $2$-step nilpotent Lie group with left-invariant Riemannian metric. Taking quotient by lattices, one obtains $\rG_2$-instantons on compact nilmanifolds.
\end{abstract}

\maketitle

\tableofcontents

\section{Introduction}

In 1998, Donaldson and Thomas  proposed the extension of  gauge theory ideas of dimension $2$, $3$ and $4$ to higher dimensional manifolds endowed with $G$-structures \cite{donaldson1998}. An important 
class of gauge fields are instantons, in which a connection satisfies a  first-order differential condition, that in many cases have variational characterisations via the Yang-Mills functional \cite{donaldson2005}. The instanton condition is formulated on a case-by-case basis according to the geometric setting, and is most often studied in the presence of a $G$-structure on the base manifold, where a certain tensor or spinor field determines algebraic conditions on the curvature of the connection. For example, on a $7$-manifold $M$ with $\rG_2$-structure defined by a positive $3$-form $\varphi$, a connection $\nabla$ on a vector bundle $E\rightarrow M$ is a \emph{$\rG_2$-instanton} if its curvature $R^\nabla$ satisfies the equation,
\begin{equation}\label{eq: instaton_equation_intro}
	R^\nabla\wedge \star \varphi=0,
\end{equation}where $\star$ denotes the Hodge dual operator induced by $\varphi$.

$\rG_2$-instantons were first constructed in the case of torsion-free $\rG_2$-structures on compact and non-compact manifolds by several different methods (e.g. \cite{Clarke2014,saearp2015,walpuski2013}). This was in part motivated by the suggestion in \cite{DoSe11,donaldson1998} that instantons on such manifolds could be used to define enumerative invariants that could distinguish different connected components of the moduli space of torsion-free $\rG_2$-structures.
They later became a subject of study for $\rG_2$-structures of several different non-zero torsion types (e.g. \cite{ball2019,lotay2018,waldron2022}). 
Notably, $\rG_2$-instantons 
have also attracted significant attention in theoretical physics, particularly in heterotic supergravity adapted to a $7$-dimensional base \cite{CG-FT05,dlOG21,OssaLaSv15,Iva10,LotaySEarp,Str86} in which the heterotic $\rG_2$-system requires a conformally coclosed $\rG_2$-structure to be coupled to a gauge field by a quadratic {\it Bianchi} condition.  
The gaugino variation present in this system requires the curvature of a certain connection to be of instanton type. Therefore,  constructing $\rG_2$-instantons is a first step to build solutions to the full heterotic $\rG_2$-system.

When working with coclosed $\rG_2$-structures $\varphi$, a natural connection to test the instanton condition is the {\em characteristic connection} of $\varphi$, that is, a metric connection $\nabla$ with skew-symmetric torsion for which $\nabla\varphi=0$ \cite{Friedrich2001}. Such a connection exists and is unique in a slightly more general setting than the coclosed one, namely, whenever $\d \star\varphi=\theta\wedge\star\varphi$ holds for some 1-form $\theta$. The characteristic connection was proven to be a $\rG_2$-instanton for nearly parallel $\rG_2$-structures \cite{Harland2012}, but this is not true in general.

In this paper, we study the instanton equation for the characteristic connection of left-invariant  $\rG_2$-structures on Lie groups. It is already known that this connection is an instanton on the Heisenberg Lie group of dimension 5 times an abelian factor,  and the Heisenberg Lie group of dimension 7 \cite{FIUV11}; in both cases $\varphi$ is taken to be purely coclosed. Also, an $\SU(3)$-instanton on the 6-dimensional free 2-step nilpotent Lie group on 3-generators $N_{2,3}$ lifts to the characteristic connection, which becomes $\rG_2$-instanton on $\R\times N_{2,3}$ \cite{IvIv05}; in this case $\varphi$ is not coclosed. Despite these examples, and to the best of our knowledge, there has not been yet considered a systematic study of the instanton condition for the characteristic connection on families of Lie groups endowed with left-invariant coclosed $\rG_2$-structures.

The goal of this paper is to contribute in this direction. We restrict ourselves to the family of (non-abelian) 2-step nilpotent Lie groups, and equip them with a left-invariant coclosed $\rG_2$-structure $\varphi$. Notice that  the left-invariant metric induced by $\varphi$  is not Einstein \cite{Mil76}, thus previous results for nearly-parallel manifolds do not apply here. 
Our techniques build on ideas used by the second named author in collaboration with Moroianu and Raffero \cite{dBMR} to study left-invariant coclosed (and purely-coclosed) $\rG_2$-structures on 2-step nilpotent Lie groups. We recall that any such  Lie group admits a coclosed $\rG_2$-structure, unless it is irreducible and has 2-dimensional commutator \cite{BFF18,dBMR}. To give our study a little more generality, we consider a 1-parameter family of connections $\nabla^\lambda$, including the characteristic connection as $\lambda=1$ (see \eqref{eq:np}). Then, we ask ourselves: when does $\nabla^\lambda$ satisfy the $\rG_2$-instanton condition \eqref{eq: instaton_equation_intro}?

The results we obtain are threefold. On the one hand, we get that the parameter $\lambda$ gives no generality to the problem. In fact, we show that if $\nabla^\lambda$ is an instanton, then $\lambda$ must be 1 and therefore we must pay attention to the characteristic connection only (see Lemma \ref{lm: instanton_condition_dim1} and Proposition \ref{pro:caracg2}).

On the other hand we obtain precise characteristics of the possible $\rG_2$-structures $\varphi$ for which the characteristic connection is an instanton. Indeed, we show that if this is the case and the Lie group has commutator subgroup of dimension 1, then $\varphi$ is purely coclosed. By contrast, if  this subgroup is of dimension $\geq 2$, then $\varphi$ is not purely coclosed.

Finally, we obtain results on the structure of the Lie group $G$ itself. To state this properly, let us denote by $\mg$ the Lie algebra of $G$. We show that if $G$ admits a coclosed $\rG_2$-structure for which $\nabla$ is a $\rG_2$-instanton, then $G$ is either the Heisenberg Lie group of dimension 7, or the one of dimension 5 times $\R^2$, or it is the quaternionic Heisenberg, or else the group $\R\times N_{2,3}$ mentioned above. 

The main result of the paper is to give the full classification of 2-step nilpotent Lie algebras (and thus simply connected Lie groups) and coclosed $\rG_2$-structures on them for which $\nabla$ is a $\rG_2$-instanton:
\begin{teo} 
On a simply connected $2$-step nilpotent Lie group $G$ with Lie algebra $\mg$, there exists a left-invariant coclosed $\rG_2$-structure $\varphi$ for which the characteristic connection is a $\rG_2$-instanton if and only if  there exists a basis $\{e^1, \ldots, e^7\}$ of $\mg^*$ such that $\varphi$ is given by \eqref{eq:gcvar} and the structure constants in that basis satisfy one of the following conditions:
\begin{enumerate}
    \item $\d e^i=0$ for $i=1, \ldots, 6$ and $\d e^7=a (e^{12}-e^{56})+b(e^{34}-e^{56})$, for some $a,b$ non simultaneously zero.
    \item $\d e^i=0$ for $i=1, \ldots, 4$ and $\dd e^5=\nu (e^{13}-e^{24})$, $\dd e^6=\nu(-e^{14}-e^{23})$, $\dd e^7=\nu(e^{12}+e^{34})$, for some $\nu\neq 0$.
     \item $\d e^i=0$ for $i=1, \ldots, 4$ and $\dd e^5=-2\nu e^{24}$, $\dd e^6=-2\nu e^{23}$, $\dd e^7=2\nu e^{34}$, for some $\nu\neq 0$.
\end{enumerate}
\end{teo}

We should point out that the $\rG_2$-instantons in (1) are precisely the ones presented by Fernández et.~al.~ \cite{FIUV11}. Besides, Agricola et.~al.~show that the metric Lie groups in (2) are naturally reductive and their canonical connection as such coincides with the characteristic connection of the $\rG_2$-structure given in the theorem \cite{AFS15}. They also prove that the holonomy of this connection is $\su(2)$, so putting these facts altogether one could deduce that they are $\rG_2$-instantons, even if it is not explicitly mentioned there. As far as we are aware, the instantons in (3) are new in the literature. Indeed, the latter cannot be obtained from the construction given by Ivanov and Ivanov \cite{IvIv05} since they produce instantons for non-coclosed $\rG_2$-structures.  Apart from the construction of these examples, and linking them together as part of the same theory,  we believe that the classification obtained as the direct part of this theorem is the most relevant contribution of this paper. 

As a consequence of this result, we are able to show that the relation with natural reductivity of case (2) is actually common to all the cases, thus giving an extra geometrical feature of the connection defining $\rG_2$-instantons.
\begin{teo}
Let $G$ be a simply connected $2$-step nilpotent Lie group of dimension $7$, let $\varphi$ be a left-invariant coclosed $\rG_2$-structure and denote by $g$ and $R$, respectively, the left-invariant metric induced by $\varphi$ and its Riemann curvature tensor. If the characteristic connection $\nabla$ defined by $\varphi$ is a $\rG_2$-instanton, then $\nabla T=\nabla R=0$ and therefore $(G,g)$ is naturally reductive and $\nabla$ is its canonical connection.
\end{teo}
We note that the converse of this theorem does not hold. Indeed, the Heisenberg Lie group of dimension 3 times a 4-dimensional abelian factor is naturally reductive but the characteristic connection of no left-invariant coclosed $\rG_2$-structure is an instanton, since it does not appear in the classification.

The Lie groups admitting instantons  listed above  possess bases with rational coefficients  and therefore  they admit lattices as well \cite{MAL}. In addition, since the $\rG_2$-structure and its canonical connection are left-invariant, they descend to the quotient by this latttice, thus leading to $\rG_2$-instantons on compact (nil)manifolds. 

Now, we give a brief account of the contents of this paper. Section \ref{sec:prelim} is aimed at fixing notation and gives the preliminaries for the rest of the paper. In Section \ref{sec:cc1}, we study the case when $G$ has commutator of dimension 1. For such Lie groups, we exploit the fact that the $\rG_2$-structure induces an $\SU(3)$-structure on the orthogonal complement of the commutator. This allows us to guarantee the existence of a basis adapted to the coclosed $\rG_2$-structure {\em and} to the Lie algebra structure (see Lemma \ref{lm:can1}). Using this basis we show that a necessary and sufficient condition for $\nabla^\lambda$ to be a $\rG_2$-instanton is that $\lambda=1$ and $\varphi$ is purely coclosed. 

Section \ref{sec:cc2} deals with the case of commutator of dimension $\geq 2$. It is worth noting that since $\dim G=7$, its commutator is of dimension at most $3$. The focus of the first part of this section is to show that, if the commutator is of dimension 3 and $\nabla^\lambda$ is an instanton for a coclosed $\rG_2$-structure $\varphi$, then $\varphi$ calibrates the commutator. This is a key step for the second part of the section, where we show that when a Lie algebra admits a coclosed $\rG_2$-structure for which $\nabla^\lambda$ is an instanton, then there is a basis adapted {\em to both} the $\rG_2$-structure and the Lie algebra structure. We finish the section with a characterization of Lie algebras, $\rG_2$-structures and parameters $\lambda$ for which $\nabla^\lambda$ is an instanton in Proposition \ref{pro:caracg2}; in particular, we get $\lambda=1$ as mentioned before. The proof of this proposition requires the computation of the curvature components of $\nabla^\lambda$, which are involved and require extra notations  so we placed them in Section \ref{sec:appadapt} to ease the reading of the article. Finally, in Section \ref{sec:class} we include the classification result mentioned above, with the full description of the examples appearing in the theorem. Moreover, the relation with the naturally reductive property is explained. 

\smallskip 

\noindent {\bf Acknowledgements:} This research was supported by MATHAMSUD Regional Program 21-MATH-06. AM was funded by the São Paulo Research Foundation (Fapesp) [2021/08026-5]. AC would like to acknowledge the hospitality of IMECC-Unicamp during his visit to Campinas. He is partially supported by the grants {BRIDGES ANR--FAPESP ANR-21-CE40-0017} and {Projeto CAPES - PrInt UFRJ 88887.311615/2018-00}.  
 VdB is funded by Fapesp grant [2021/09197-8]. The authors are  grateful to Udhav Fowdar for useful discussions on the topic.
 
\section{ \texorpdfstring{$\rG_2$}{G2}-Structures on 2-Step Nilpotent Lie Groups}\label{sec:prelim}

\subsection{Preliminaries on \texorpdfstring{$\rG_2$}{G2}-Structures}

A $\rG_2$-structure on a $7$-manifold $M$ is given by a differential $3$-form $\varphi\in\Omega^3(M)$ that is at each point isomorphic to the $3$-form 
\[
\varphi_0=e^{127}+e^{347}+e^{567}+e^{135}-e^{146}-e^{236}-e^{245}, 
\]
on $\mathbb{R}^7$, where $\{e^1, \ldots, e^7\}$ is the dual to the standard basis of $\R^7$. 
The $\rG_2$-structure $\varphi$ 
canonically determines a Riemannian metric $g=g_\varphi$ and a volume form $\vol_\varphi$ on $M$ by 
\begin{equation}\label{eq:xyvol}
	(X\lrcorner\varphi)\wedge(Y\lrcorner\varphi)\wedge\varphi = 6g_\varphi(X,Y){\vol_\varphi}, \quad  X,Y\in\mathfrak X(M),	
\end{equation}
and hence the Levi-Civita connection $\nabla^g$ of $g$ and Hodge star operator $\star$ on $M$. We denote by $\psi=\star\varphi$, the Hodge dual $4$-form  of $\varphi$.
A $\rG_2$-structure $\varphi$ induces decompositions of the bundles $\Lambda^k$ into sums of subbundles, with factors corresponding to irreducible representations of $\rG_2$. For example, 
\begin{equation}\label{eq: decomposition_of_forms}
	\Lambda^2=\Lambda^2_7\oplus\Lambda^2_{14} \qquad \Lambda^3=\Lambda^3_1\oplus\Lambda^3_7\oplus\Lambda^2_{27},
\end{equation}
where $\Lambda^k_l$ has rank $l$ and each submodule can be characterized  algebrically using the forms $\varphi$ and $\psi$. For instance, 
\begin{equation}\label{eq: definition_Omega}
\Lambda^2_{14} =\{\alpha\in \Lambda^2: \alpha\wedge\varphi=-\star\alpha\}=\{\alpha\in \Lambda^2: \alpha\wedge\psi=0\}\simeq \mg_2.
\end{equation}

The decompositions of $\Lambda^4$ and $\Lambda^5$ are similar.
Hence,
the exterior derivatives of $\varphi$ and $\psi$ can be written in terms of the \emph{torsion forms} $\tau_j\in \Omega^j=\Gamma(\Lambda^j)$ (for $j=0,1,2,3$) which are defined by (see \cite{Bryant2006})
\begin{equation}\label{eq: exterior_derivative_varphi}
	\dd \varphi = \tau_0\psi+3\tau_1\wedge \varphi+\star\tau_3 \qandq \dd \psi=4\tau_1\wedge\psi+\tau_2\wedge\varphi.
\end{equation}
Explicitly, the torsion forms can be written in terms of $\varphi$ and $\psi$: 
\begin{align}\label{eq: torsion_forms}
	\tau_0=&\frac{1}{7}\star(\varphi\wedge \dd\varphi), \quad \tau_1=\frac{1}{12}\star(\varphi\wedge\star\dd\varphi)=\frac{1}{12}\star(\psi\wedge\star\dd\psi),\\ \nonumber
	\tau_2=& -\star\dd\psi+4\star(\tau_1\wedge\psi), \quad \tau_3=\star\dd\varphi-\tau_0\varphi-3\star(\tau_1\wedge\varphi).
\end{align}
$\rG_2$-structures can be classified into $16$-classes according to which of the four torsion components are identically zero.
For example, $\varphi$ is said to be \emph{torsion free} if $\tau_j=0$ for all $j$. By a well-known theorem of Fernandez and Gray \cite{Fernandez1982}, this is equivalent to $\nabla^{g}\varphi=0$, and hence to a reduction of the Riemannian holonomy group. The $\rG_2$-structure $\varphi$ is {\em coclosed} if $\tau_1=\tau_2=0$, and is {\em purely coclosed} if additionally $\tau_0=0$. We note that this is equivalent to $\varphi\wedge \d\varphi=0=\d\star\varphi$.

The conditions of coclosedness and pure coclosedness are of particular interest for several reasons. Firstly, Friedrich and Ivanov \cite{Friedrich2001} show that if $\tau_2=0$, there exists a unique metric connection $\nabla$ on the tangent bundle for which the torsion $T(X,Y,Z)=g(X,T_\nabla(Y,Z))$ is totally skew-symmetric and for which $\nabla\varphi=0$. In this case, the torsion $3$-form is given by 
\begin{equation}\label{eq:torcontau1}
	T=\frac16 \star( \d\varphi\wedge\varphi)\varphi-\star \d\varphi+4\star(\tau_1\wedge\varphi).
\end{equation}
Secondly, coclosed $\rG_2$-structures are those that are required for solutions to the heterotic equations of motion in $7$-dimensions. This system of equations, coming from string theory, asks for a spinor field to be parallel with respect to a metric connection with skew-torsion. A \emph{dilaton condition} corresponds in the $7$-dimensional case to the torsion term $\tau_1$ being an exact $1$-form $\tau_1=\d f$. The function $f$ being constant therefore results in the induced $\rG_2$-structure $\varphi$ being coclosed. We note that the conditions $\tau_0=0$, and $\tau_0$ equal to a non-zero constant, correspond in the physics literature to different compactification ans\"atze (see \cite{FrIv03,OssaLaSv15}). 

Let $(M,\varphi)$ be a $7$-manifold equipped with a $\rG_2$-structure and $(E,\nabla)$ a vector bundle with connection on $M$. The condition for  $\nabla$ to be a $\rG_2$-instanton \eqref{eq: instaton_equation_intro} is equivalent to its curvature $R^\nabla$ to take values in the subbundle $\Lambda^2_{14}\otimes\mathrm{End}(E)$. A natural example of a $\rG_2$-instanton is the Levi-Civita connection of a torsion-free $\rG_2$-structure. The condition $\nabla^g\varphi=0$ gives a holonomy reduction for $\nabla^g$, and implies that $R^{\nabla^g}$ takes values in $\Lambda^2\otimes\mathfrak{g}_2\subseteq \Lambda^2\otimes\mathfrak{so}(TM)$, while the symmetries of the curvature tensor of a torsion-free connection imply that it lies in $\Lambda^2_{14}\otimes\mathfrak{so}(TM)$. We note that if $\nabla$ is a connection  on $TM$ with non-vanishing torsion for which $\nabla\varphi=0$, it is not necessarily the case that $\nabla$ satisfies the instanton condition. 

With this in mind, we consider the 1-parameter family  of connections $\nabla^\lambda$ on $TM$, $\lambda\in\R$, given by 
\begin{eqnarray}
\label{eq:np}
	\lela \nabla_X^\lambda Y,Z\rira =\lela \nabla_X^gY,Z\rira +\frac{\lambda}{2} \,T(X,Y,Z), \qquad X,Y,Z\in\mathfrak X(M),
\end{eqnarray}
for $T$ as in \eqref{eq:torcontau1}, where $M$ is a $7$-dimensional $2$-step nilpontent Lie group, equipped with a coclosed $\rG_2$-structure $\varphi$, and ask for which values of $\lambda$ is $\nabla^\lambda$ a $\rG_2$-instanton. We note immediately that $\nabla^\lambda$ is compatible with the metric $g=g_\varphi$, that the torsion of $\nabla^\lambda$ is totally skew-symmetric, and equal to $\lambda T$, however $\varphi$ is not parallel with respect to $\nabla^\lambda$, for $\lambda\neq1$.

\subsection{Geometry of Riemannian 2-step Nilpotent Lie Groups}\label{sec:prelim2st}

Let $G$ be a connected $7$-dimensional (non-abelian) 2-step nilpotent Lie group endowed with a left-invariant $\rG_2$-structure. That is, a non-degenerate 3-form $\varphi\in\Omega^3(G)$ that is invariant under left-translations of $G$. Note that the Riemannian metric $g$ and the volume form induced by $\varphi$ on $G$ by the formula \eqref{eq:xyvol} are also left-invariant. Denote by $\mathfrak{g}$ the Lie algebra of $G$. We can thus see $\varphi$ as an element  of $\Lambda^3\mathfrak{g}^*$, $g$ as an inner product on $\mg$, etc...

The $\rG_2$-structure $\varphi$ is closed or coclosed  when itself or its Hodge dual $\psi$, respectively,  are closed with respect to the Chevalley-Eilenberg Lie algebra differential $\d:\Lambda^\bullet\mg^*\to\Lambda^\bullet \mg^*$.

Left-invariant coclosed $\rG_2$-structures on 2-step nilpotent Lie algebras where studied in \cite{BFF18}. It is known that any such a Lie algebra admits coclosed $\rG_2$-structure, except when $\mg$ is irreducible and has 2-dimensional commutator. In addition, purely coclosed $\rG_2$-structures appear on any Lie algebra admitting coclosed ones, except when $\mg$ is the direct sum of the 3-dimensional Heisenberg Lie algebra with $\R^4$ \cite{dBMR}.

From now on we assume that $\varphi$ is coclosed. We shall introduce some basic facts about the geometry of the underlying Riemannian Lie group $(G,g=g_\varphi) $ by means of algebraic features of its metric Lie algebra $(\mg,g)$ \cite{EB}.

On the Lie algebra $\mg$, the commutator $\mg'$ and the center $\mz$  are defined by
\begin{equation*}
\mg'=\langle\{[x,y]:x,y\in\mg\}\rangle, \qquad \mz=\{z\in\mg:[z,x]=0\mbox{ for all }x\in\mg\},
\end{equation*}where  $\langle Y\rangle$ denotes the subspace spanned by $Y$. The  2-step nilpotent condition on $\mg$ is equivalent to $0\neq \mg'\subset \mz$. We denote by $\ma$ the (possibly trivial) abelian factor of $\mg$, given by
\begin{equation}\label{eq:ma}
\ma:=\mz\cap (\mg')^\bot.
\end{equation} 

Let $\mc$  be a fixed subspace verifying $\mg'\subset\mc\subset\mz$ which is, in fact, an ideal of $\mg$. Let $\mq:=\mc^\perp\subseteq \mr:=(\mg')^\perp$. 
Define $j:\mc\to\mathfrak{so}(\mq)$ by, for
$x,y\in\mq$ and $z\in\mc$, 
\begin{equation}\label{eq:jc}
\lela j(z)x,y\rira=\lela z,[x,y]\rira. 
\end{equation} 
We see that $\ker j=(\mg')^\bot\cap\mc$ and thus the restriction $j|_{\mg'}$ is injective. In particular, if $\mc=\mz$,  $\ker \,(j:\mz\to\so(\mz^\bot))=\ma$ because of \eqref{eq:ma}. 
Even though, strictly speaking, different subspaces $\mc$ define different maps $j$, they are all related. To picture this, consider the maps $j_\mc:\mc\to\so(\mq)$ and $j_{\mg'}:\mg'\to\so(\mr)$ defined by \eqref{eq:jc} for $\mc$ and $\mg'$, respectively. For any $z\in\mg'$, $j_{\mc}(z)\in\so(\mq)$ whilst  $j_{\mg'}(z)\in\so(\mq\oplus(\mr\cap \mc))$. However, it is easy to check that $j_{\mg'}(z)$ is the extension  of $j_{\mc}(z)$ by zero on $\mr\cap \mc$. 
Due to this close relationship between the maps $j_\mc$ and $j_{\mg'}$ and in order to avoid heavy notations, we will denote all possible maps (for all possible choices of $\mc$) by $j$, and we will just specify their domain and/or target when necessary.

The structure coefficients of a Lie algebra are determined by the differential map $\d:\mg^*\to\Lambda^2\mg^*$.
Since $\mg$ is 2-step nilpotent, the image of $\d:\mg^*\to\Lambda^2\mg^*$ is contained in $\Lambda^2\mr^*$. Given $z\in\mg'$, let $z^\flat\in\mg^*$ denote the linear map $z^\flat(\cdot)=\lela z,\cdot\rira$.
Under the identification $\Lambda^2\mr^*$ with $\so(\mr)$ via the metric $g$, the differential $\d z^\flat\in\Lambda^2\mr^*$ corresponds to $-j(z)\in\so(\mr)$; this follows immediatly from Cartan's formula and \eqref{eq:jc} for $\mc=\mg'$.

The Levi-Civita connection $\nabla^g$ of $(G,g)$ can be viewed as a map $\nabla^g:\mg\to\so(\mg)$, such that  $x\in \mg$ is taken to $\nabla^g_x\in\so(\mg)$ which, by Koszul formula, verifies
\begin{equation}\label{eq:Koszul}
\lela \nabla^g_xy,z\rira =\frac12\left(\lela [x,y],z\rira+\lela [z,x],y\rira+\lela [z,y],x\rira\right),  \quad \forall x,y,z\in\mg.
\end{equation}
More specifically,
\begin{equation}\label{eq:nabla}\left\{
\begin{array}{ll}
\nabla^g_u v=\frac12 \,[u,v] & \mbox{ if } u,v\in\mr,\\
\nabla^g_u z=\nabla^g_zu=-\frac12 j(z)u & \mbox{ if } u\in\mr,\,z\in\mg',\\
\nabla^g_z z'=0& \mbox{ if } z, z'\in\mg'.
\end{array}\right.
\end{equation}

For every $x\in\mg$, $\nabla^g_x$ can be seen as a skew-symmetric endomorphism of $\mg$; actually, from  \eqref{eq:nabla} we have
\begin{equation}
\nabla_u^g=\frac12 (\ad_u-\ad_u^*),\qquad \nabla_z^g=-\frac12 j(z)\qquad \forall u\in \mr,\,z\in \mg',
\end{equation} where here $\ad_u^*$ is the metric adjoint of $\ad_u$ and  $j(z)$ denotes the extension to $\mg$, by zero on $\mg'$, of the map $j(z)\in\so(\mr)$ defined in \eqref{eq:jc}. In fact, as any metric connection, $\nabla^g$ can be represented as an $\so(n)$-valued 1-form $\Psi_0$ on $\mg$. Namely, given an orthonormal basis $\{e_i\}_{i=1}^7$ of $\mg$ and its dual basis $\{e^i\}_{i=1}^7$, one has
\begin{equation}\label{eq:psiLC}
\Psi_0=\sum_{i=1}^7 e^i\otimes\nabla_{e_i}^g.
\end{equation}

The coclosed $\rG_2$-structure $\varphi$ defines a 1-parameter family of metric connections $\nabla^\lambda$ via the formula \eqref{eq:np}. 
Notice that  $\nabla^0=\nabla^g$  is not an instanton since otherwise  $(G,g)$ would have holonomy contained in $\rG_2$ and thus it would be Ricci flat  \cite{Bon66}, contradicting the fact that $\mg$ is not abelian \cite[Theorem 2.4]{Mil76}. Consequently, when studying the instanton condition for $\nabla^\lambda$, we will assume $\lambda\neq 0$.

The connection 1-form corresponding to $\nabla^\lambda$ in \eqref{eq:np} is
\begin{align}\label{eq:psilam}
	\Psi_\lambda=\sum_{i=1}^7 e^i\otimes\nabla_{e_i}^\lambda=\Psi_0+\frac{\lambda}{2}T,
\end{align}  where $T$ in \eqref{eq:torcontau1} becomes
\begin{equation}\label{eq:tor}
T=\frac16 \star( \d\varphi\wedge\varphi)\varphi-\star \d\varphi,
\end{equation} due to closedness of $\psi$, 
and the corresponding $\so(n)$-valued curvature 2-form reads
\begin{align*}
	R^\lambda=\dd \Psi_\lambda+\Psi_\lambda\wedge\Psi_\lambda.
\end{align*}
Explicitly, given $x,y\in \mg$, we have the endomorphism 
\begin{align}\label{eq:Rlam}
	R^\lambda_{x,y}=&\dd \Psi_\lambda(x,y)+(\Psi_\lambda\wedge\Psi_\lambda )(x,y)=-\Psi_\lambda([x,y])+\Psi_\lambda(x)\Psi_\lambda(y)-\Psi_\lambda(y)\Psi_\lambda(x)\\
	=&-\nabla^\lambda_{[x,y]}+\nabla^\lambda_x\nabla^\lambda_y-\nabla^\lambda_y\nabla^\lambda_x.\nonumber         \end{align}

Alternately, for $\alpha,\beta\in\{1, \ldots, 7\}$, the entry of the curvature  $(R^\lambda)^\beta_\alpha\in \Lambda^2\mg^*$ can be computed in the basis $\{e_i\}_{i=1}^7$ as follows: 
\begin{equation}\label{eq:Rlamab}
(R^\lambda)^\alpha_\beta=\sum_{1\leq j<k\leq 7}\lela R^\lambda_{e_j,e_k}e_\beta,e_\alpha\rira  e^j\wedge e^k
\end{equation}
and $R^\lambda$ satisfies the $\rG_2$-instanton condition \eqref{eq: instaton_equation_intro} if and only if $(R^\lambda)^\alpha_\beta\wedge\psi=0$, for all $\alpha,\beta\in\{1, \ldots, 7\}$.

\section{The Instanton Condition for 2-step Nilpotent Lie Algebras \texorpdfstring{$\mg$}{g} Satisfying \texorpdfstring{$\dim\mg'=1$}{dim g'=1}.}\label{sec:cc1}

In this section $\mg$ is a $2$-step nilpotent Lie algebra with $1$-dimensional commutator $\mg'$. Let $\varphi$ be a $\rG_2$-structure on $\mg$ and consider the orthogonal splitting $\mg=\mr\oplus\mg'$ with respect to $g=g_\varphi$, where $\mr=(\mg')^\perp$. Given a unit vector $z\in \mg'$, the structure equations of $\mg$ are completely determined by $A=:-j(z)\in \so(\mr)$ satisfying \eqref{eq:jc} for $\mc=\mg'$. Then the 6-dimensional vector space $\mr$ is endowed with the flat $\SU(3)$-structure $(h,\omega,\rho_\pm)$ given by
\begin{equation}\label{eq:phiomrho}
	\varphi=\omega\wedge z^\flat+\rho_+, \quad \psi=\frac{\omega^2}{2}+\rho_-\wedge z^\flat \quad \text{and} \quad g=h+z^\flat\otimes z^\flat.
\end{equation}
The $\rG_2$-structure $\varphi$ is coclosed if and only if $A\in \mathfrak{u}(3):=\{B\in \so(\mr): BJ=JB\}$, where $J$ is defined by $\omega(\cdot,\cdot)=h(J\cdot,\cdot)$. 

\begin{lm}\label{lm:can1}
	For every coclosed $\rG_2$-structure $\varphi$ on $\mg$ there exists an orthonormal basis $\{e_1,\ldots,e_7\}$ of $\mg$, such that:
	\begin{equation}\label{eq: G2_struct_dim_1}
		\varphi=e^{127}+e^{347}+e^{567}+e^{135}-e^{146}-e^{236}-e^{245}
	\end{equation}
	and
	\begin{equation}\label{eq: structure_equation_dim_1}
		\d e^i=0, \quad i=1, \ldots, 6, \quad \d e^7=ae^{12}+be^{34}+ce^{56},
	\end{equation}for some $a,b,c\in\R$,  not all simultaneously zero.
\end{lm}

\begin{proof}By definition of $\rG_2$-structures there exists an orthonormal basis $\{\tilde e^1, \ldots,\tilde e^7\}$ of $\mg^*$ such that, on this basis,  $\varphi$ has the form \eqref{eq: G2_struct_dim_1}.
	Using that $\rG_2$ acts transitively on the 6-sphere, we can assume that $\tilde e_7\in \mg'$. Let $(\omega,\rho_\pm)$ be the induced $\SU(3)$-structure on $\mr=\langle\{\tilde e_1,\ldots,\tilde e_6\}\rangle$.
	Since $\varphi$ is coclosed, the matrix $A:=-j(\tilde e_7)\in\so(\mr)$ commutes with $J$ \cite[Proposition 4.1]{dBMR}. In particular, there exists a unitary matrix $U\in \U(3)$ such that on the basis $\{U\tilde e_1,\ldots, U\tilde e_6\}$, $A$ and $J$ have real canonical form
	\begin{align}\label{eq: normal_form_of_A}
		A=\left(\begin{array}{cc}
			0 & -a \\
			a & 0 
		\end{array}\right)\oplus\left(\begin{array}{cc}
			0 & -b \\
			b & 0 
		\end{array}\right)\oplus\left(\begin{array}{cc}
			0 & -c \\
			c & 0 
		\end{array}\right), \qandq 	J=\left(\begin{array}{cc}
			0 & -1 \\
			1 & 0 
		\end{array}\right)^{\oplus 3}.
	\end{align}
	 Set $u:=\overline{\det U}$ and consider $v\in \C$ satisfying $v^3=u$, thus we have that $V:=v U\in \SU(3)$. Then,  $\{e_1:=V\tilde e_1,\ldots,e_6:=V\tilde e_6,e_7:=\tilde e_7\}$ is a basis where $\varphi$ has the form \eqref{eq: G2_struct_dim_1} and their differentials satisfy \eqref{eq: structure_equation_dim_1}. Since $\mg$ is not abelian, $a,b,c$ cannot be simultaneously zero.
\end{proof}
Using the basis given in the lemma, and combining \eqref{eq: structure_equation_dim_1} with \eqref{eq:nabla} and \eqref{eq:psiLC}, we get that the  $\so(7)$-valued 1-form  on $\mg$ of the Levi-Civita connection of $(\mg,g)$ is
\begin{equation}\label{eq: LC_connection_dim_1}
	\Psi_0=
	\frac{1}{2}\left(\begin{array}{ccc|c}
		& & & \\
		& e^7\otimes A & & -A\be \\
		& & &  \\ \hline 
		& \be^TA & & 0
	\end{array}\right),
\end{equation}
where $\be^T=(e^1,e^2,\ldots,e^6)$ and $\be^TA:=(ae^2,-ae^1,be^4,-be^3,ce^6,-ce^5)$. By \eqref{eq: G2_struct_dim_1} and \eqref{eq: structure_equation_dim_1}, we get
\begin{equation}\label{eq:dvarphi1}
	\d\varphi=(a+b)e^{1234}+(a+c)e^{1256}+(b+c)e^{3456},
\end{equation} and thus, setting $\mu=a+b+c$, we further obtain
\begin{equation}\label{eq:pccg1}
    \star(\d\varphi\wedge\varphi)=2\mu.
\end{equation} In particular, $\varphi$ is purely coclosed  (see \eqref{eq: torsion_forms}) if and only if $\mu=0$ (see also \cite[Proposition 4.1]{dBMR}).

Moreover, by \eqref{eq:dvarphi1}, the torsion $3$-form $T$ associated with the coclosed $\rG_2$-structure via \eqref{eq:tor} is 
\begin{eqnarray}\label{eq:Tdim1}
	T
 &=&\frac{\mu}{3}\left(e^{135}-e^{146}-e^{236}-e^{245}\right)\nonumber 
 +\frac{1}{3}\left((a-2b-2c)e^{12}-(2a-b+2c)e^{34}-(2a+2b-c)e^{56}\right)\wedge e^7.\nonumber 
\end{eqnarray}
We can write this $3$-form as the $\so(7)$-valued 1-form:
\begin{equation}\label{eq: Torsion_dim_1}
	T=
	\frac{\mu}{3}\left(\begin{array}{ccc|c}
		& & & \\
		&-2e^7\otimes J+M & & -2J\be \\
		& & & \\ \hline
		& 2\be^TJ & & 0
	\end{array}\right)+
	\left(\begin{array}{ccc|c}
		& & & \\
		& e^7\otimes A & & A\be \\ 
		& & & \\ \hline
		& -\be^TA & &  0
	\end{array}\right),
\end{equation}
where 
\begin{equation*}
	M=\left(\begin{array}{cc|cc|cc}
		0 & 0 & -e^5 & e^6 & e^3 & -e^4   \\
		0 & 0 & e^6 & e^5 & -e^4 & -e^3   \\ \hline
		e^5&-e^6 & 0 & 0 & -e^1 & e^2  \\
		-e^6& -e^5& 0 & 0 & e^2 & e^1   \\ \hline
		-e^3& e^4& e^1&-e^2 & 0 & 0 \\
		e^4 & e^3& -e^2&-e^1 & 0 &0 
	\end{array}\right).
\end{equation*}

Let $\nabla^\lambda$ be the metric connection defined by $\varphi$ via formula \eqref{eq:np}, namely $\nabla^\lambda=\nabla^g+\frac{\lambda}{2}T$, where $T$ is as in \eqref{eq:Tdim1}. Its connection 1-form is $\Psi_\lambda=\Psi_0+\frac{\lambda}{2}T$, as given in \eqref{eq:psilam}.

\begin{lm}\label{lm: instanton_condition_dim1}
	Let $\mg$ be a 2-step nilpotent Lie algebra with structure equation \eqref{eq: structure_equation_dim_1} and $\rG_2$-structure \eqref{eq: G2_struct_dim_1}. If the connection $\nabla^\lambda$  is a $\rG_2$-instanton then, $\lambda=1$ and $\mu=0$. 
\end{lm}

\begin{proof}	Assume that $\nabla^\lambda$ is a $\rG_2$-instanton \eqref{eq: instaton_equation_intro}. In terms of the curvature components \eqref{eq:Rlamab}, this is equivalent to $(R^\lambda)^\alpha_\beta\wedge\psi=0$ for all $\alpha,\beta=1, \ldots, 7$.
For $\alpha\neq 7$, using that $AJ=JA$, we have
\begin{eqnarray*}
		(R^\lambda)_\alpha^7&=&\sum_{j<k}\lela (\Psi_\lambda\wedge \Psi_\lambda)_{e_j,e_k}e_\alpha,e_7\rira e^j\wedge e^k\\
  &=&\sum_{j<k}\lela\left( (\Psi_\lambda)_{e_k}( \Psi_\lambda)_{e_j}-(\Psi_\lambda)_{e_j}( \Psi_\lambda)_{e_k}\right)e_7,e_\alpha\rira e^j\wedge e^k\\
 &=& \sum_{j,k}
 \lela (\lambda-1)(\Psi_\lambda)_{e_k}A\be_{e_j}-\frac{2\mu\lambda}{3}(\Psi_\lambda)_{e_k}J\be_{e_j},e_\alpha\rira e^j\wedge e^k\\
		&=& \sum_{j,k}\left(\frac{1}{2}(\lambda^2-1)e^7(e_k)(A^2\be_{e_j})_\alpha-\frac{\mu\lambda}{3}(\lambda+1)e^7(e_k)(AJ\be_{e_j})_\alpha-\frac{\mu\lambda}{3}(\lambda-1)e^7(e_k)(JA\be_{e_j})_\alpha\right. \\
&& \ \ \ \ \ \ \ 		\left.
-\frac{2\mu^2\lambda^2}{9}e^7(e_k)(\be_{e_j})_\alpha   
		+\frac{\mu\lambda}{6}\left((\lambda-1)(M_{e_k}A\be_{e_j})_\alpha-\frac{2\mu\lambda}{3}(M_{e_k}J\be_{e_j})_\alpha\right)\right)e^j\wedge e^k \\
 &=& \sum_{j<k}-\left(\frac{1}{4}(\lambda^2-1)((e^7\wedge A^2\be)_{e_j, e_k})_\alpha-\frac{\mu\lambda^2}{3}((e^7\wedge JA\be)_{e_j,e_k})_\alpha-\frac{\mu^2\lambda^2}{9}((e^7\wedge\be)_{e_j,e_k})_\alpha\right. \\
&& \ \ \ \ \ \ \ 		\left.   
		+\frac{\mu\lambda}{12}\left((\lambda-1)((M\wedge A\be)_{e_j,e_k})_\alpha-\frac{2\mu\lambda}{3}((M\wedge J\be)_{e_j,e_k})_\alpha\right)\right)e^j\wedge e^k\\
  &=& -\frac{1}{4}(\lambda^2-1)(e^7\wedge A^2\be)_\alpha+\frac{\mu\lambda^2}{3}(e^7\wedge AJ\be)_\alpha+\frac{\mu^2\lambda^2}{9}(e^7\wedge\be)_\alpha  \\
&& \ \ \ \ \ \ \ 		
		-\frac{\mu\lambda}{12}\left((\lambda-1)(M\wedge A\be)_\alpha-\frac{2\mu\lambda}{3}(M\wedge J\be)_\alpha\right),   
	\end{eqnarray*} 
		where $(M\wedge A\be)_\alpha$ and $(M\wedge J\be)_\alpha$ are the components of the vectors 
	\begin{align*}
		M\wedge A\be=&\left(\begin{array}{c}
			-(b+c)(e^{36}+e^{45})\\
			-(b+c)(e^{35}-e^{46})\\
			(a+c)(e^{16}+e^{25})\\
			(a+c)(e^{15}-e^{26})\\
			-(a+b)(e^{14}+e^{23})\\
			-(a+b)(e^{13}-e^{24})
		\end{array}\right) \qandq M\wedge J\be=\left(\begin{array}{c}
			-2(e^{36}+e^{45})\\
			-2(e^{35}-e^{46})\\
			2(e^{16}+e^{25})\\
			2(e^{15}-e^{26})\\
			-2(e^{14}+e^{23})\\
			-2(e^{13}-e^{24})
		\end{array}\right).\end{align*}
	Hence, for the entries $(R^\lambda)^1_7$, $(R^\lambda)^3_7$ and $(R^\lambda)^5_7$ we have:
	\begin{align*}
		(R^\lambda)^1_7=&-\frac{1}{36}\left[\left((3a(\lambda-1)-2\mu\lambda)(3a(\lambda+1)-2\mu\lambda)\right)e^{17}+\mu\lambda\left(\mu(\lambda+3)+3a(\lambda-1)\right)(e^{36}+e^{45})\right],\\
		(R^\lambda)^3_7=&-\frac{1}{36}\left[\left((3b(\lambda-1)-2\mu\lambda)(3b(\lambda+1)-2\mu\lambda)\right)e^{37}-\mu\lambda\left(\mu(\lambda+3)+3b(\lambda-1)\right)(e^{16}+e^{25})\right],\\
		(R^\lambda)^5_7=&-\frac{1}{36}\left[\left((3c(\lambda-1)-2\mu\lambda)(3c(\lambda+1)-2\mu\lambda)\right)e^{57}+\mu\lambda\left(\mu(\lambda+3)+3c(\lambda-1)\right)(e^{14}+e^{23})\right].
	\end{align*}
	Using \eqref{eq: G2_struct_dim_1}, we have that the equations $(R^\lambda)^k_7\wedge\psi=0$, for $k=1,3,5$ are equivalent to
	 either $\lambda=-1$ or the following system of equations
	\begin{align}\label{eq: F_k7_wedge_psi}
		2\mu^2\lambda-2\mu a\lambda+3a^2(\lambda-1)=0,\\ \nonumber
		2\mu^2\lambda-2\mu b\lambda+3b^2(\lambda-1)=0,\\
		2\mu^2\lambda-2\mu c\lambda+3c^2(\lambda-1)=0. \nonumber
	\end{align}If the latter holds, adding the three equations we obtain $4\mu^2\lambda+3(a^2+b^2+c^2)(\lambda-1)=0$, which implies $0< \lambda\leq 1$. 
 
 Now, we compute the entries $(R^\lambda)^1_2,(R^\lambda)^3_4$ and $(R^\lambda)^5_6$ of the curvature:
	\begin{align*}
		(R^\lambda)^1_2=& \d(\Psi_\lambda)^1_2+\sum_{k=1}^7(\Psi_\lambda)^1_k\wedge (\Psi_\lambda)^k_2\\
		=&-\frac{1}{6}(3a(\lambda+1)-2\mu\lambda)(ae^{12}+be^{34}+ce^{56})+\frac{\mu^2\lambda^2}{18}(e^{34}+e^{56})-\frac{1}{36}(3a(\lambda-1)-2\mu\lambda)^2e^{12}.
	\end{align*}
	Likewise we get:
	\begin{align*}
		(R^\lambda)^3_4=&-\frac{1}{6}(3b(\lambda+1)-2\mu\lambda)(ae^{12}+be^{34}+ce^{56})+\frac{\mu^2\lambda^2}{18}(e^{12}+e^{56})-\frac{1}{36}(3b(\lambda-1)-2\mu\lambda)^2e^{34},\\
		(R^\lambda)^5_6=&-\frac{1}{6}(3c(\lambda+1)-2\mu\lambda)(ae^{12}+be^{34}+ce^{56})+\frac{\mu^2\lambda^2}{18}(e^{12}+e^{34})-\frac{1}{36}(3c(\lambda-1)-2\mu\lambda)^2e^{56}.
	\end{align*}
	Using \eqref{eq: G2_struct_dim_1}, the equations $(R^\lambda)^k_{k+1}\wedge\psi=0$ for $k=1,3,5$ are equivalent to
	\begin{align*}
		-3(\lambda-1)^2a^2+4\mu^2\lambda+2(2\lambda^2-5\lambda-3)a\mu=0,\\
		-3(\lambda-1)^2b^2+4\mu^2\lambda+2(2\lambda^2-5\lambda-3)b\mu=0,\\
		-3(\lambda-1)^2c^2+4\mu^2\lambda+2(2\lambda^2-5\lambda-3)c\mu=0.
	\end{align*}
	The last system of equations implies the equation 
	\begin{equation}\label{eq: second_eq}
		-3(\lambda-1)^2(a^2+b^2+c^2)+2(2\lambda^2-\lambda-3)\mu^2=0.
	\end{equation}

 By the above, we know that $\lambda$ is either $-1$ or in the interval $0<\lambda\leq 1$. However, if $\lambda=-1$, \eqref{eq: second_eq} becomes $-12(a^2+b^2+c^2)-8\mu^2=0$, which cannot hold since $a,b,c$ do not vanish simultaneously. In addition,  $0< \lambda\leq 1$ implies $-25/8<2\lambda^2-\lambda-3\leq -2$ which, together with \eqref{eq: second_eq}, gives $\lambda=1$ and $\mu=0$.
\end{proof}

\begin{teo}\label{teo:gprim1}
 	Let $\mg$ be a $2$-step nilpotent Lie algebra with structure equation \eqref{eq: structure_equation_dim_1} and $\rG_2$-structure \eqref{eq: G2_struct_dim_1}. The connection $\nabla_\lambda$  is a $\rG_2$-instanton if and only if $\lambda=1$ and $\varphi$ is purely-coclosed.
\end{teo}

\begin{proof}
	From Lemma \ref{lm: instanton_condition_dim1}, a necessary condition for $R^\lambda\wedge\psi=0$ is $\lambda=1$ and $\mu=0$. The latter is equivalent to $\varphi$ purely-coclosed (see \eqref{eq:dvarphi1}). We shall prove that these conditions are also sufficient. 
	
	In fact,  fixing $\lambda=1$ and $\mu=0$ and using  \eqref{eq: LC_connection_dim_1} and \eqref{eq: Torsion_dim_1} we get that the connection  $1$-form $\Psi_1=\Psi_0+\frac{1}{2}T$ in \eqref{eq:psilam} becomes:
	\begin{align}
		\Psi_1=\left(\begin{array}{c|c}
			e^7\otimes A &  0\\ \hline
			0& 0
		\end{array}\right)\in \mg^\ast\otimes\su(3)\subset  \mg^\ast\otimes\mg_2,\label{eq:gprim1psi}
	\end{align}
	and the corresponding curvature $2$-form is:
	\begin{align}
		R^1=\left(\begin{array}{c|c}
			\d e^7\otimes A & 0 \\  \hline
			0& 0
		\end{array}\right)\in \Lambda^2\mg^\ast\otimes\su(3)\subset  \Lambda^2\mg^\ast\otimes\mg_2,\label{eq:gprim1R}
	\end{align}
	where $\d e^7=a(e^{12}-e^{56})+b(e^{34}-e^{56})$. Finally, notice that $\omega$ in \eqref{eq:phiomrho} is $\omega=e^{12}+e^{34}+e^{56}$ due to \eqref{eq: G2_struct_dim_1}, so a simple computation shows $\d e^7\wedge \omega^2=0$ and thus $R^1\wedge\psi=0$ by \eqref{eq:phiomrho}.
\end{proof}

\section{The Instanton Condition for 2-Step Nilpotent Lie Algebras \texorpdfstring{$\mg$}{g} Satisfying \texorpdfstring{$\dim\mg'\geq 2$}{dim g'>=2}. }\label{sec:cc2}

Recall from Section \ref{sec:prelim2st} that the commutator subalgebra $\mg'$ of a 2-step nilpotent Lie algebra of dimension 7 has dimension at most three. Having already worked with $\dim \mg'=1$, in this section we focus on the instanton question for $\nabla^\lambda$ for the remaining cases $\dim\mg'=2$ and $3$. We will show that,  in these cases, the instanton condition is closely related to the calibratedness of some central subspace.  First we show that if $\dim\mg'=3$ and the connection $\nabla^\lambda$ defined by a coclosed $\rG_2$-structure $\varphi$ is an instanton, then  $\varphi$  calibrates $\mg'$. 
Note that  this necessary condition excludes certain coclosed invariant $\rG_2$-structures from consideration (see \cite[Ex. 4.15]{dBMR}).

Next, we proceed to the general case $\dim\mg'\geq 2$. We show that for $\nabla^\lambda$ to be an instanton, $\lambda=1$, the coclosed structures are necessarily \emph{not} purely coclosed, that $\dim\mg'=3$ and the Lie algebra structure constants satisfy precise conditions.

\subsection{A Necessary Condition}\label{subsec: non-calibrated}
Let $\mg$ denote a 2-step nilpotent Lie algebra of dimension 7 with 3-dimensional commutator $\mg'$ and let $\varphi$ be a coclosed $\rG_2$-structure on $\mg$ inducing a metric $g=g_\varphi$. Fix an orthonormal  basis  $\{e_5,e_6,e_7\}$ of $\mg'$. Using the transitive action of $\rG_2$ on ordered pairs of orthonormal vectors in $\R^7$, one can get an orthonormal basis $\{\tilde e_1, \ldots,\tilde e_7\}$ of $\mg$ such that $\tilde e_6=e_6$ and $\tilde e_7=e_7$ and
\begin{equation}\label{eq:canphi}
	\varphi= \tilde e^{127}+\tilde e^{347}+\tilde e^{567}+\tilde e^{135}-\tilde e^{146}-\tilde e^{236}-\tilde e^{245}.
\end{equation}

Since $e_5\bot \tilde e_6,\tilde e_7$, we have $e_5\in\langle \{\tilde e_1, \ldots, \tilde e_5\}\rangle$. Moreover, the stabilizer of $\tilde e_6$ and $\tilde e_7$ fixes $\tilde e_5$ as well and acts as ${\rm SU}(2)$ on $\tilde \mr:=\langle \{\tilde e_1, \ldots, \tilde e_4\}\rangle$. Hence, we can assume that the component of $e_5$ on $\tilde\mr$ is proportional to $\tilde e_3$. This means that there exist $r,s\in\R$ such that $r^2+s^2=1$ and $e_5=r\tilde e_3+s\tilde e_5$. 

Recall that $\varphi$ is said to {\em calibrate} a 3-dimensional subspace $\mc$ of $\mg$ if there exists an orthonormal basis $\{z_1, z_2,z_3\}$ of $\mc$ such that $\varphi(z_1, z_2,z_3)=\pm1$. We refer the reader to  \cite{Harvey1982} for the theory of calibrations. Since $\varphi(e_5,e_6,e_7)= s$, one gets that $\varphi$ calibrates $\mg'$ if and only if $s=1$ and $r=0$.

Set $e_i:=\tilde e_i$ for $i=1,2,4$ and  $e_3:=s\tilde e_3-r \tilde e_5$, so that   $\{e_1, \ldots, e_7\}$ is a  $g$-orthonormal basis. In the dual basis, we have
\[
\left\{\begin{array}{l}
	\tilde e^3=s e^3+r e^5\\
	\tilde e^5=-r e^3+s e^5
\end{array}
\right.,
\]so the $\rG_2$-structure in \eqref{eq:canphi} in the new basis becomes
\begin{equation}\label{eq:G2_structure_rs}
	\varphi= e^{127}+e^{135}-e^{146}-r(e^{234}+e^{256}+e^{367}+e^{457})-s(e^{236}+e^{245}-e^{347}-e^{567}).
\end{equation}
We readily compute
\begin{equation}\label{eq:hodge_G2_structure_rs}
	\psi=e^{2357}-e^{2467}+e^{3456}-r(e^{1236}+e^{1245}-e^{1347}-e^{1567})+s(e^{1234}+e^{1256}+e^{1367}+e^{1457}).
\end{equation}
Notice that $\varphi$ induces the following $\SU(3)$-structure on $\langle \{e_1\}\rangle^\bot$:
\begin{equation}\label{eq:SU3_structure_rs}
	\omega=e^{27}+e^{35}-e^{46} \qandq \rho=s(e^{234}+e^{256}+e^{367}+e^{457})-r(e^{236}+e^{245}-e^{347}-e^{567})
\end{equation}
which verifies
\begin{equation*}
	\frac{\omega^3}{3!}=\frac{1}{4}\rho\wedge J^\ast\rho=e^{234567},
\end{equation*}
where $J$ is the almost complex structure given by $Je_2=e_7, Je_3=e_5$ and $Je_6=e_4$.
We note that since $e_5, e_6,e_7$ span $\mg'$, the structure constants satisfy
\begin{equation}\label{eq:deialrs}\left\{
\begin{array}{rcll}
     \d e^i&= &0, &\quad i=1, \ldots,4  \\
    \d e^i & =:&\alpha_i\in \Lambda^2\langle \{e_1, \ldots ,e_4\}\rangle^*,&\quad i=5,6,7.
\end{array}\right.
\end{equation}

\begin{lm}\label{lem: structure_equations_n37A}
     If $\varphi$ does not calibrate $\mg'$, then there exists an orthonormal basis $\{e_1,\dots,e_7\}$ of $\mg$ such that $\varphi$ is given by \eqref{eq:G2_structure_rs} and $\mg$ has structure constants 
	\begin{equation}\label{eq:struct_eqs_n37A_simplified}
		\dd e^5=d_3e^{13}, \quad \dd e^6=-d_4e^{14} \qandq \dd e^7=d_2e^{12},
	\end{equation}for some $d_2,d_3,d_4\in\R\bs\{0\}$.
\end{lm}

\begin{proof}
Assume that $\varphi$ does not calibrate $\mg'$ so that $r\neq 0$ in \eqref{eq:G2_structure_rs}. Differentiating \eqref{eq:hodge_G2_structure_rs} and using that $\varphi$ is coclosed and \eqref{eq:deialrs}, we get
\begin{eqnarray}
	0=\dd \psi &=&e^{237}\wedge \alpha_5-e^{235}\wedge \alpha_7
	-e^{247}\wedge \alpha_6+e^{246}\wedge \alpha_7+e^{346}\wedge \alpha_5-e^{345}\wedge \alpha_6\nonumber\\
	&&+re^1\wedge(-e^{67}\wedge \alpha_5+e^{57}\wedge \alpha_6
	-e^{56}\wedge \alpha_7-e^{34}\wedge \alpha_7+e^{24}\wedge \alpha_5+e^{23}\wedge \alpha_6)\label{eq:dvar}\\
	&&+se^1\wedge(e^{26}\wedge \alpha_5-e^{25}\wedge \alpha_6
	+e^{37}\wedge \alpha_6-e^{36}\wedge \alpha_7+e^{47}\wedge \alpha_5-e^{45}\wedge \alpha_7).\non
\end{eqnarray}
Taking the contraction of this 5-form with two elements in the commutator gives, for instance, 
\[0=e_6\lrcorner e_5\lrcorner\dd \psi=r\,e^1\wedge \alpha_7.\]
Since $r\neq 0$, this implies $e^1\wedge \alpha_7=0$. Similarly, one obtains $0=e^1\wedge \alpha_5=e^1\wedge \alpha_6$. Hence, we get that 
\begin{equation}\label{eq:ali1w}
    \alpha_{i+4}=e^1\wedge\left(a_{i1}e^2+a_{i2}e^3+a_{i3}e^4\right), \qquad i=1,2,3
\end{equation} for some $(a_{ij})_{i,j=1}^3$ and thus \eqref{eq:dvar} 
implies
\begin{equation}\label{eq:3eq}
	e^{23}\wedge \alpha_5-e^{24}\wedge \alpha_6=0,\quad
	e^{24}\wedge \alpha_7+e^{34}\wedge \alpha_5=0,\quad
	e^{23}\wedge \alpha_7
	+e^{34}\wedge \alpha_6=0.
\end{equation}
Replacing each $\alpha_{i+4}$ in \eqref{eq:3eq} by its expression in \eqref{eq:ali1w}, we obtain
\begin{equation}
	a_{11}=a_{32},\quad a_{21}=-a_{33},\quad a_{13}=-a_{22}.
\end{equation}
Consequently, the Lie algebra structure coefficients in terms of the orthonormal basis are:
\begin{equation}\label{eq:struct_eqs_n37A}
	\left\{\begin{array}{l}
		\dd e^5=e^1\wedge (a_{12}e^3+a_{13}e^4+a_{11}e^2),\\
		\dd e^6=e^1\wedge (-a_{13}e^3+a_{23}e^4+a_{21}e^2),\\
		\dd e^7=e^1\wedge  (a_{11}e^3-a_{21}e^4+a_{31}e^2).
	\end{array}\right.
\end{equation}

Consider the following matrices $M=(m_{ij})$ and $S=(s_{ij})$
	\begin{equation}
 		S=\left(\begin{array}{ccc}
			a_{12} & -a_{13} & a_{11} \\
			-a_{13} & -a_{23} & a_{21} \\
			a_{11} & a_{21} & a_{31}
		\end{array}\right),\qquad
  		M=\left(\begin{array}{ccc}
			0&1&0\\
            0&0&-1\\
            1&0&0
		\end{array}\right).
	\end{equation} Then \eqref{eq:struct_eqs_n37A} can be rewritten as $\d e^{k+4}= e^1\wedge\sum_{j,l=1}^3 s_{kj}m_{jl}e^{l+1}= e^1\wedge\sum_{l=1}^3 (SM)_{kl}e^{l+1}$, for $k=2,3,4$.
 Let $ Q=(q_{ij})\in \SO(3)$ be such that $S=QDQ^t$, with $D$ a diagonal matrix $D={\rm diag}(b_2,b_3,b_4)$.  Set 
 \begin{equation}\label{eq:deftil} f^1:=e^1,\quad
f^{i+4}:=\sum_{k=1}^3q_{ki}e^{k+4}=\sum_{k=1}^3(Q^t)_{ik}e^{k+4}, \quad f^{i+1}:=\sum_{l=1}^3=(M^tQ^tM)_{il}e^{l+1}, \qquad i=1,2,3.
 \end{equation}
 Then, for $i=1, 2,3$,
 \begin{eqnarray*}
     \dd f^{i+4}&=&e^1\wedge  \sum_{k=1}^3q_{ki} \d e^{k+4}
     =e^1\wedge  \sum_{j,l=1}^3(SQ)_{ji}m_{jl}e^{l+1}=e^1\wedge  \sum_{j,l=1}^3(QD)_{ji}m_{jl}e^{l+1}\\
     &=&b_{i}e^1\wedge  \sum_{l=1}^3(Q^tM)_{il}e^{l+1}=
     b_{i}e^1\wedge  \sum_{l,j=1}^3m_{ij}(M^tQ^tM)_{jl}e^{l+1}=b_{i}e^1\wedge  \sum_{j=1}^3m_{ij}f^{j+1}.
      \end{eqnarray*}
Hence we get formulas \eqref{eq:struct_eqs_n37A_simplified} for some $d_2,d_3,d_4\in\R$. Notice that if some $d_i$ vanishes, then $\dim\mg'<3$, so they are all non-zero. Since $M$ and $Q$ are orthogonal,  $\{f^1, \ldots f^7\}$ is an orthonormal basis. Moreover, the orthogonal map $A:\mg\to\mg$ verifying $A^*e^i= f^i$ as in \eqref{eq:deftil} preserves the subspaces $\langle\{ e_1\}\rangle$ and $\langle\{ e_1\}\rangle^\perp$, and also, it  preserves the $\SU(3)$-structure \eqref{eq:SU3_structure_rs} and the $\rG_2$-structure $\varphi$. Moreover, applying $A^*$ to both sides in \eqref{eq:G2_structure_rs} we get
\[
\varphi= f^{127}+f^{135}-f^{146}-r(f^{234}+f^{256}+f^{367}+f^{457})-s(f^{236}+f^{245}-f^{347}-f^{567}).
\]
\end{proof}

\begin{remark}By using the basis in Lemma \ref{eq:instanton_n37A} one can show that if $\mg$ admits a coclosed $\rG_2$-structure
not calibrating $\mg'$ then $\mg$ is isomorphic to the Lie algebra $\mn_{7,3,A}$ in Gong's classification \cite{Gon98}.  This fact was already shown through a slightly different argument in \cite[Proposition 4.14]{dBMR}.

\end{remark}
Now we consider the affine connection $\nabla^\lambda$  on $\mg$  defined by $\varphi$ as in \eqref{eq:np} and we study the instanton condition on it. 

\begin{pro}\label{pro:instcal} If $\nabla^\lambda$ is an instanton, then  $\varphi$ calibrates $\mg'$.
\end{pro}
\begin{proof}
	Let us suppose that $\varphi$ does not calibrate $\mg'$. Then, there is a basis satisfying the conditions in Lemma \ref{lem: structure_equations_n37A}. By \eqref{eq:struct_eqs_n37A_simplified}, the connection 1-form of the Levi-Civita connection in that basis is 
\begin{equation}
	\Psi_0=\frac{1}{2}\left(\begin{array}{ccc}
		0 & -\zeta_1^T & -\zeta_2^T\\
		\zeta_1 &  \mathbf{0}_{3\times 3} & -B^T\\
		\zeta_2 & B & \mathbf{0}_{3\times 3} 
	\end{array}\right), 
\end{equation}
for
\begin{equation}
	\zeta_1=\left(\begin{array}{c}
		d_2 e^7 \\ d_3e^5 \\ -d_4e^6
	\end{array}\right), \quad \zeta_2=\left(\begin{array}{c}
		d_3 e^3 \\ -d_4e^4 \\ d_2e^2
	\end{array}\right) \qandq B=\left(\begin{array}{ccc}
		0 & -d_3e^1 & 0 \\
		0 & 0 & d_4e^1 \\
		-d_2e^1 & 0 & 0
	\end{array}\right)
\end{equation}
Setting $\mu=d_2+d_3+d_4$, the torsion 3-form \eqref{eq:tor} is:
\begin{align*}
	T=&\frac{\mu}{3}\left(e^{127}+e^{135}-e^{146}\right)-r\left(\frac{\mu}{3}e^{234}-\left(\frac{2\mu}{3}-d_2\right)e^{256}-\left(\frac{2\mu}{3}-d_3\right)e^{367}-\left(\frac{2\mu}{3}-d_4\right)e^{457}\right)\\
	&-s\left(\left(\frac{\mu}{3}-d_4\right)e^{236}+\left(\frac{\mu}{3}-d_3\right)e^{245}-\left(\frac{\mu}{3}-d_2\right)e^{347}+\frac{2\mu}{3}e^{567}\right).
\end{align*}

 Therefore, the connection 1-form  $\Psi_\lambda$ of $\nabla^\lambda$, which is given  \eqref{eq:psilam},  has the following curvature terms (see \eqref{eq:Rlamab})
	\begin{align}\label{eq:curvatures_12-13-14}\nonumber
		(R^\lambda)^1_2=p(d_2)e^{12}&+s\lambda\left(p_1(d_2,d_3,d_4)e^{34}+p_2(d_2,d_3,d_4)e^{56}\right)\\ \nonumber
		&-r\lambda\left(p_3(d_2,d_3,d_4)e^{36}+p_4(d_2,d_3,d_4)e^{45}\right)\\
		(R^\lambda)^1_3=p(d_3)e^{13}&-s\lambda\left(p_1(d_3,d_4,d_2)e^{24}-p_2(d_3,d_4,d_2)e^{67}\right)\\ \nonumber
		&+r\lambda\left(p_3(d_3,d_4,d_2)e^{47}+p_4(d_3,d_4,d_2)e^{26}\right)\\ \nonumber		(R^\lambda)^1_4=p(d_4)e^{14}&+s\lambda\left(p_1(d_4,d_2,d_3)e^{23}+p_2(d_4,d_2,d_3)e^{57}\right)\\ \nonumber
		&+r\lambda\left(p_3(d_4,d_2,d_3)e^{25}-p_4(d_4,d_2,d_3)e^{37}\right)
	\end{align}
	where $	p(x)=-\frac{\mu^2}{36}\lambda^2-\frac{3}{4}x^2$ and
	\begin{align}\label{eq:p_functions}\nonumber
		p_1(x,y,z)=&\frac{1}{36}\left((3x-\mu)\mu\lambda+6y^2-3xy-6yz-3xz+6z^2\right)\\
		p_2(x,y,z)=&\frac{1}{36}\left((3x-\mu)\mu\lambda+3y^2+3xy-12yz+3xz+3z^2\right)\\ \nonumber
		p_3(x,y,z)=&\frac{1}{36}\left((3x-\mu)\mu\lambda+6y^2-3xy+9yz+3xz+3z^2\right)\\ \nonumber
		p_4(x,y,z)=&\frac{1}{36}\left((3x-\mu)\mu\lambda+3y^2+3xy+9yz-3xz+6z^2\right).
	\end{align}
	Taking the wedge product between \eqref{eq:hodge_G2_structure_rs} and \eqref{eq:curvatures_12-13-14}, we have
	\begin{align}\label{eq: R1_wedge_ast_varphi} \nonumber
		(R^\lambda)^1_2\wedge\psi=&\left(s^2\lambda(p_1+p_2)+r^2\lambda(p_3+p_4)+p\right)e^{123456}+sr\lambda(p_1+p_2-p_3-p_4)e^{134567}\\
		(R^\lambda)^1_3\wedge\psi=&\left(s^2\lambda(p_1+p_2)+r^2\lambda(p_3+p_4)+p\right)e^{123467}-sr\lambda(p_1+p_2-p_3-p_4)e^{124567}\\ \nonumber
		(R^\lambda)^1_4\wedge\psi=&\left(s^2\lambda(p_1+p_2)+r^2\lambda(p_3+p_4)+p\right)e^{123457}+sr\lambda(p_1+p_2-p_3-p_4)e^{123567}.
	\end{align}
	From \eqref{eq:p_functions} and \eqref{eq: R1_wedge_ast_varphi}, the condition $(R^\lambda)^1_j\wedge\psi=0$ for $j=2,3,4$ implies
	\begin{align}\label{eq:instanton_n37A}\nonumber
		2d_2\mu\lambda^2+3s^2\lambda(d_3-d_4)^2+3r^2\lambda(d_3+d_4)^2-\mu^2\lambda^2-9d_2^2=0, \quad  s\lambda d_3d_4=0\\
		2d_3\mu\lambda^2+3s^2\lambda(d_4-d_2)^2+3r^2\lambda(d_4+d_2)^2-\mu^2\lambda^2-9d_3^2=0, \quad s\lambda d_2d_4=0\\ \nonumber
		2d_4\mu\lambda^2+3s^2\lambda(d_2-d_3)^2+3r^2\lambda(d_2+d_3)^2-\mu^2\lambda^2-2d_4^2=0, \quad s\lambda d_2d_3=0.
	\end{align} 
	As pointed out in \S \ref{sec:prelim2st}, $\lambda\neq 0$ since $\mg$ is not abelian. Moreover, since $d_2,d_3,d_4$ are all nonzero, the above equation implies $s=0$.  We will show that this takes us to a contradiction. 
 
 When $s=0$, \eqref{eq:instanton_n37A} becomes
	\begin{align}\label{eq:instanton_n37A_s=0}\nonumber
		2d_2\mu\lambda^2+3\lambda(\mu-d_2)^2-\mu^2\lambda^2-9d_2^2=0\\
		2d_3\mu\lambda^2+3\lambda(\mu-d_3)^2-\mu^2\lambda^2-9d_3^2=0\\ \nonumber
		2d_4\mu\lambda^2+3\lambda(\mu-d_4)^2-\mu^2\lambda^2-9d_4^2=0.
	\end{align} 
	Similarly, we have the curvature terms
	\begin{align}\label{eq:curvatures_35-46-27}\nonumber
		(R^\lambda)^2_7=q(d_2)e^{27}&-\lambda^2\left(q_1(d_2,d_3,d_4)e^{35}-q_1(d_2,d_4,d_3)e^{46}\right)\\ 
		(R^\lambda)^3_5=q(d_3)e^{35}&-\lambda^2\left(q_1(d_3,d_2,d_4)e^{27}-q_1(d_3,d_4,d_2)e^{46}\right)\\ \nonumber
		(R^\lambda)^4_6=q(d_4)e^{46}&+\lambda^2\left(q_1(d_4,d_2,d_3)e^{27}+q_1(d_4,d_3,d_2)e^{35}\right)
	\end{align}
	where $	q(x)=-\frac{\mu^2}{36}\lambda^2+\frac{1}{4}x^2$ and $q_1(x,y,z)=\frac{1}{36}\left(4x^2-xy-xz+4y^2-yz-5z^2\right)$. Thus, 
	taking the wedge of the forms in \eqref{eq:G2_structure_rs} and \eqref{eq:curvatures_35-46-27}, we get
	\begin{align}\label{eq: R2_wedge_ast_varphi} \nonumber
		(R^\lambda)^2_7\wedge\psi=&(q(d_2)-\lambda^2(q_1(d_2,d_3,d_4)+q_1(d_2,d_4,d_3)))e^{234567}=\frac{d_2^2}{4}(1-\lambda^2)e^{234567}\\
		(R^\lambda)^3_5\wedge\psi=&(q(d_3)-\lambda^2(q_1(d_3,d_2,d_4)+q_1(d_3,d_4,d_2)))e^{234567}=\frac{d_3^2}{4}(1-\lambda^2)e^{234567}\\ \nonumber
		(R^\lambda)^4_6\wedge\psi=&(\lambda^2(q_1(d_4,d_2,d_3)+q_1(d_4,d_3,d_2))-q(d_4))e^{234567}=\frac{d_4^2}{4}(\lambda^2-1)e^{234567}.
	\end{align}
	The non-trivial vanishing condition on \eqref{eq: R2_wedge_ast_varphi} implies that $\lambda=\pm 1$ and using this in \eqref{eq:instanton_n37A_s=0}, we obtain the equation
	\begin{equation}\label{eq:d2_d3_d4=0}
		3\lambda(\mu^2+d_2^2+d_3^2+d_4^2)-\mu^2-9(d_2^2+d_3^2+d_4^2)=0.
	\end{equation}
	For $\lambda=-1$, the equation \eqref{eq:d2_d3_d4=0} only has solution for $d_2=d_3=d_4=0$, but these coefficients are all nonzero. Hence $\lambda=1$. But in this case,  \eqref{eq:d2_d3_d4=0} becomes
	\begin{align*}
		0=3(\mu^2+d_2^2+d_3^2+d_4^2)-\mu^2-9(d_2^2+d_3^2+d_4^2)=4(d_2d_3+d_2d_4+d_3d_4-d_2^2-d_3^2-d_4^2),
	\end{align*}
 for which one can readily check that the only solutions are $d_2=d_3=d_4=0$, leading us to a contradiction. So $\nabla^\lambda$ is not an instanton.
\end{proof}

\subsection{Necessary and Sufficient Conditions}\label{ss:nsc}
Let $\mg$ be a $2$-step nilpotent Lie algebra of dimension $7$ such that $\dim\mg'=2$ or $3$ and admitting  coclosed $\rG_2$-structures $\varphi$. We will study the instanton condition for the connection $\nabla^\lambda$ defined by $\varphi$ through  formula \eqref{eq:np}.

Assume for the moment that the coclosed $\rG_2$-structure $\varphi$ on $\mg$ admits a calibrated subspace $\mc$ satisfying $\mg'\subset\mc\subset\mz$. Then, there exists an orthonormal basis  $\{e_1, \ldots, e_7\}$  of $\mg$ such that (see \cite{Harvey1982})  $\mc=\langle \{e_5,e_6,e_7\}\rangle$ and 
\begin{equation}\label{eq:gcvar}
   \varphi=e^{127}+e^{347}+e^{567} + e^{135} -e^{146} -e^{236}-e^{245}.
\end{equation} Let $\mq$ denote the orthogonal of $\mc$. 
Since $\mq\bot\mg'$, the structure coefficients of $\mg$ satisfy
\begin{equation}\label{eq:deial}\left\{
\begin{array}{rcll}
     \d e^i&= &0, &\quad i=1, \ldots,4  \\
    \d e^i & =:&\alpha_i\in \Lambda^2\mq^*,&\quad i=5,6,7.
\end{array}\right.
\end{equation}

Consider the orientation induced by  $e^{1234}$ on $\mq$. Then,  the space of $2$-forms $\Lambda^2\mq^*$ decomposes into the eigenspaces of the Hodge operator, namely, into the space of self-dual and anti-self-dual forms. Hence, we can decompose further
\begin{equation}\label{eq:strbody}
	\alpha_i=\d e^i=\alpha_i^++\alpha_i^-  ,\quad i=5,6,7,
\end{equation}where $\alpha^+_i$ (resp.~$\alpha_i^-$) is the self-dual (resp.~anti-self-dual) component of $\alpha_i$ in $\Lambda^2\mq^*$. Notice that $\{\alpha_5,\alpha_6,\alpha_7\}$ span a subspace of dimension $\dim\mg'$ inside $\Lambda^2\mq^*$. 

Consider the basis of self-dual $2$-forms in $\mq$ given by
\begin{equation}
	\label{eq:sigmabasis}
	\sigma_1^+=e^{13}-e^{24}, \quad \sigma_2^+=-e^{14}-e^{23},\quad \sigma_3^+=e^{12}+e^{34}.
\end{equation} Since $\varphi$ is given by \eqref{eq:gcvar}, we can use this basis to write 
\begin{equation}
	\label{eq:phibody}
	\varphi=\sigma_1^+\wedge e^5+\sigma_2^+\wedge e^6+\sigma_3^+\wedge e^7+e^{567},
\end{equation}
so that the Hodge dual of $\varphi$ becomes
\begin{equation}\label{eq:stphi}
	\psi=e^{1234}+\sigma_1^+\wedge e^{67}+\sigma_2^+\wedge e^{75}+\sigma_3^+\wedge e^{56}.
\end{equation}

By \eqref{eq:strbody} and \eqref{eq:phibody} we get
\begin{equation}\label{eq:pwdp}
	\varphi\wedge \d\varphi=2(\sigma_1^+\wedge \alpha_5^++\sigma_2^+\wedge \alpha_6^++\sigma_3^+\wedge \alpha_7^+)\wedge e^{567}
\end{equation}
and, since $\varphi$ is coclosed,
\begin{equation}
	0=\d\psi=(\sigma_2^+\wedge \alpha_7^+-\sigma_3^+\wedge \alpha_6^+)\wedge e^{5}+(\sigma_3^+\wedge \alpha_5^+-\sigma_1^+\wedge \alpha_7^+)\wedge e^{6}+(\sigma_1^+\wedge \alpha_6^+-\sigma_2^+\wedge \alpha_5^+)\wedge e^{7}.
\end{equation}
This implies that the matrix $S=(s_{ij})_{i,j=1}^3$, with components $s_{ij}:=\frac12 g(\sigma_i^+,\alpha_{j+4}^+)$, is symmetric (see also \cite{dBMR}). Hence, in the space of self-dual forms in $\Lambda^2\mq^*$, we have
\begin{equation}\label{eq:cc}
	\alpha_5^+=s_{11}\sigma_1^++s_{12}\sigma_2^++s_{13}\sigma_3^+,\quad 
	\alpha_6^+=s_{12}\sigma_1^++s_{22}\sigma_2^++s_{23}\sigma_3^+,\quad
	\alpha_7^+=s_{13}\sigma_1^++s_{23}\sigma_2^++s_{33}\sigma_3^+.
\end{equation} 

The next result shows that $\mg$ admits a basis on which the matrix $S=(s_{ij})_{i,j=1}^3$ corresponding to the $\rG_2$-structure is diagonal.

\begin{pro}\label{pro:Sdiag}
Let $\varphi$ be a coclosed $\rG_2$-structure on a Lie algebra $\mg$ with $\dim\mg'\geq2$. If $\varphi$ calibrates a subspace $\mc$ satisfying $\mg'\subset\mc\subset\mz$, then there exists an orthonormal basis $\{e_1, \ldots, e_7\}$ of $\mg$ such that $e_5,e_6,e_7$ span $\mc$, \eqref{eq:gcvar} holds and, in terms of the basis \eqref{eq:sigmabasis},
	\begin{equation*}
		(\d e^5)^+=d_5 \sigma_1^+, \quad (\d e^6)^+=d_6 \sigma_2^+ \quad \text{and} \quad (\d e^7)^+=d_7\sigma_3^+,
	\end{equation*}for some $d_i\in\R$, $i=5,6,7$.
	\end{pro}

\begin{proof}
Let $\{\tilde e_1, \ldots,\tilde e_7\}$ be a  basis of $\mg$ such that $\{\tilde e_5,\tilde e_6,\tilde e_7\}$ span $\mc$ and \eqref{eq:gcvar} holds. In particular, 
\begin{equation}\label{eq:philem}
		\varphi=\tilde\sigma_1^+\wedge \tilde e^5+\tilde\sigma_2^+\wedge \tilde e^6+\tilde\sigma_3^+\wedge \tilde e^7+\tilde e^{567},
	\end{equation}where $\tilde\sigma_1^+,\tilde\sigma_2^+,\tilde \sigma_3^+$ is the basis of self-dual forms in $\mc^\bot=\langle\{\tilde e_1, \ldots,\tilde e_4\}\rangle$ given in \eqref{eq:sigmabasis}
	
	According to the identification $\mc=\langle\{\tilde e_5,\tilde e_6,\tilde e_7\}\rangle$ with the imaginary quaternions $ \im \H$ given by the cross product of $\varphi$, we can write $\tilde\sigma_l^+(u,v)=\langle \tilde e_{l+4}\cdot u,v\rangle$ for $l=1,2,3$, where the product $\cdot$ on the right hand side is the quaternion product.
	
	Since $\varphi$ is coclosed, the self dual forms $\tilde\alpha_i^+:=(\d \tilde e^i)^+$, $i=5,6,7$ are completely determined by $S\tilde e_5,S\tilde e_6,S\tilde e_7\in \im\H$, respectively, where  $S$ is the symmetric matrix in \eqref{eq:cc} corresponding to $\{\tilde e_1, \ldots,\tilde e_7\}$. Namely, for each $i=5,6,7$,
	\[	 \tilde \alpha_{i}^+(u,v)=\langle S\tilde e_{i}\cdot u,v\rangle, \qquad \forall u,v\in\mc^\bot.\]
	
	Since $S$ is symmetric, it can be written as $S=P_hDP_h^t$ where $P_h\in \SO(3)$ is given by $P_h(x)=hx\bar{h}$ with $h\in \Sp(1)$, $x\in \R^3$ and $D={\rm diag}(d_5,d_6,d_7)$. Now, for $h\in \Sp(1)$ denote by $\Phi\in \rG_2$ the map $\Phi(a,b)=(\bar{h}ah,\bar{h}bh)$ for $(a,b)\in \R^7=\H\oplus\im\H\simeq\mg$ (see \cite[Ch IV, Eq. (1.9)]{Harvey1982}). By definition, $\Phi$ preserves $\mc$ and $\mc^\bot$.
 
 Consider the basis $\{e_i:=\Phi\tilde e_i\}_{i=1}^7$; in particular, for $i=5,6,7$, $e_i=P_h^t\tilde e_ i$. Using that the inner product is invariant under right and left multiplications by quaternions, for every $u,v\in \mc^\bot$, we have  
	\begin{eqnarray*}
		\tilde \alpha_{l+4}^+(u,v)&=&\langle S\tilde e_{l+4}\cdot u,v\rangle=d_l\langle h e_{l+4}\bar{h}u,v\rangle=d_l\langle e_{l+4}\bar{h}u,\bar{h}v\rangle=d_l\langle e_{l+4}\bar{h}uh,\bar{h}vh\rangle\\
		&=&d_l\sigma_l^+(\Phi u,\Phi v)=d_l\Phi^*\sigma_l^+.
	\end{eqnarray*}Hence $(\Phi^t)^\ast\tilde \alpha_{l+4}^+=d_l\sigma_l^+$. 
	
	Notice that $\Phi^*e^i=\tilde e^i$ for $i=1, \ldots, 7$. Since $\Phi$ is an isometry and thus commutes with the Hodge star operator we finally obtain
	\[d_l\sigma_l^+= (\Phi^t)^\ast\tilde\alpha_{l+4}^+=(\Phi^t)^\ast (\dd \tilde e^{l+4})^+= ((\Phi^t)^\ast \dd \tilde e^{l+4})^+=(\dd (\Phi^t)^\ast  \tilde e^{l+4})^+=(\dd e^{l+4})^+,
	\]as we wanted to show. Moreover, since $\Phi\in\rG_2$, applying $(\Phi^t)^*$ to both sides of the equality \eqref{eq:philem}, we get
	\[
	\varphi=\sigma_1^+\wedge  e^5+\sigma_2^+\wedge  e^6+\sigma_3^+\wedge  e^7+ e^{567},
	\]so $\{e_1, \ldots,e_7\}$ is the required basis.
\end{proof}

We are now ready to introduce the main results of the section.
\begin{pro}\label{pro:caracg2}
    Let $\varphi$ be a coclosed $\rG_2$-structure on $\mg$ and let $\nabla^\lambda$ be the connection defined by $\varphi$ via \eqref{eq:np}.
    The connection $\nabla^\lambda$ is a $\rG_2$-instanton if and only if $\dim\mg'=3$ and the following conditions hold
    \begin{enumerate}
\item  $\lambda=1$,
\item  there exists an orthonormal basis $\{e_1, \ldots, e_7\}$  of $\mg$ such that $\mg'=\langle\{ e_5,e_6,e_7\}\rangle$, $\varphi$ is as in \eqref{eq:gcvar} and, for some $\mu\neq 0$,
    \begin{equation}
        \label{eq:sigmu}
(\dd e^{i+4})^+=\frac{\mu}3\sigma_i^+, \quad i=1,2,3,
\end{equation}
\item  the map $j:\mg'\to\so(\mr)$ satisfies
\begin{equation}\label{eq:jcondteo}
[j(z),j(z')]=\frac23 \mu j(\tau(z)z'), \qquad \forall z,z'\in\mg',
\end{equation}where, for any $z\in\mg'$,  $\tau(z)\in\so(\mg')$ is the endomorphism corresponding to $z\lrcorner (\varphi|_{\mg'})$.
\end{enumerate}
 \end{pro}

\begin{proof}  Assume first that $\varphi$ is coclosed and $\nabla^\lambda$ is an instanton. Recall that $\mg'\subset\mz$ because of the 2-step nilpotency hypothesis. We claim that there exists a subspace $\mc$ verifying $\mg'\subset\mc\subset\mz$ calibrated by $\varphi$. In fact, 
if $\dim\mg'=3$, we can take $\mc=\mg'$ due to Proposition \ref{pro:instcal}. 
If $\dim \mg'=2$ and $\{z_1,z_2\}$ is an orthonormal basis of $\mg'$, the fact that $\rG_2$ acts transitively on ordered pairs of orthonormal vectors in $\R^7$, implies that there exists an orthonormal basis $\{\tilde e_1, \ldots, \tilde e_7\}$ of $\mg$ such that $\tilde e_6=z_1$ and $\tilde e_7=z_2$ and $\varphi$ has the form in \eqref{eq:gcvar}. In particular, $\varphi$ calibrates $\mc:=\langle  \{\tilde e_5, \tilde e_6,\tilde e_7\}\rangle\supset \mg'$.  Moreover,  by  \cite[Proposition 4.4]{dBMR}, $\tilde e_5=\tilde e_7\lrcorner \tilde e_6\lrcorner \varphi$ is central, i.e. $\tilde e_5\in\mz$ and thus $\mc\subset \mz$ as well.
This shows our claim. 

Consequently, $\mg$ and $\varphi$ satisfy Assumption \ref{as0}, so we can apply the results in Section \ref{sec:appadapt}. In addition, by Proposition \ref{pro:Sdiag}  there exists an orthonormal basis $\{e_1, \ldots, e_7\}$ of $\mg$ such that $e_5, e_6, e_7$ span $\mc$ and, if $\sigma_i^+$ are the self-dual forms in $\Lambda^2\langle\{e^1, \ldots, e^4\}\rangle$ given in \eqref{eq:sigmabasis}, then $\d e^i=0$ for $i=1,\ldots, 4$ and 
\[
(\dd e^{i+4})^+=d_{i+4}\sigma_i^+, \quad i=1,2,3,
\]for some $d_i\in\R$.

Since $\nabla^\lambda$ is an instanton, Lemma \ref{lm:inst} implies that $\dim\mg'=3$ (thus $\mg'=\mc$),  $\lambda=1$ and there is some $\mu\neq 0$  for which
\eqref{eq:sigmu} and \eqref{eq:jcondteo} hold.

For the converse, assume that $\varphi$ is a coclosed $\rG_2$ structure, $\dim\mg'=3$ and (1)--(3) hold. The existence of the basis in (2) implies that $\varphi$ calibrates $\mg'$ and thus Assumption \ref{as0} holds, so we can apply the results in Section \ref{sec:appadapt}. Clearly, (1)--(3) in the statement are equivalent to (1)--(3) Lemma \ref{lm:inst}; the converse of this lemma implies that $\nabla^\lambda$ is an instanton for $\lambda=1$.
\end{proof}

We can now give special characteristics of the coclosed $\rG_2$-structures inducing instantons on $\mg$.
\begin{cor}\label{cor:parallel} Let $\varphi$ be a coclosed $\rG_2$-structure on $\mg$ and let $\nabla^\lambda$ be the connection defined by $\varphi$ via \eqref{eq:np}. If $\nabla^\lambda$ is a $\rG_2$-instanton, then  $\varphi$ 
calibrates $\mg'$, $\varphi$ is not purely coclosed and $\nabla^\lambda\varphi=0$.
\end{cor}
\begin{proof}By Proposition \ref{pro:caracg2}, if $\nabla^\lambda$ is an instanton, then $\dim\mg'=3$ and thus $\mg'=\mc$, which is calibrated by $\varphi$. Also, by (1) in that proposition, $\lambda=1$. Recall that, within the 1-parameter family of  connections \eqref{eq:np}, $\lambda=1$ is the (unique) value of the parameter  making $\varphi$ parallel \cite{Friedrich2001}. Finally, by (2) in Proposition \ref{pro:caracg2}, $\varphi$ satisfies the conditions in Section \ref{app:diagonal}, so $\varphi$ is not purely coclosed by Corollary \ref{cor:notpcc}.
\end{proof}

\begin{remark} Given a coclosed $\rG_2$-structure on a 7-manifold $M$, the torsion $T$ in \eqref{eq:tor} of the connection $\nabla^1$ is not, in general, parallel with respect to $\nabla^1$ (see \cite{Friedrich2001,Fr07}). 
In the setting of $2$-step nilpotent Lie groups, this can be seen, for instance, on the seven-dimensional extensions of the Lie algebra called $\mh_3$ in \cite{FIUV11} (which is not the Heisenberg as in our notation). There, $\nabla^1$ is denoted by $\nabla^+$, and it is shown that $\nabla^1T\neq0$  if the commutator is of dimension $>1$.

We will show in the next section, that the instanton condition on 2-step nilpotent Lie algebras forces not only that $\lambda=1$, but also that $T$ is parallel with respect to the instanton connection.
\end{remark}

\section{Examples and  Classification}\label{sec:class}

This section contains the classification of the Lie algebras and the coclosed $\rG_2$-sturctures on them for which the connection $\nabla^1$  is a $\rG_2$-instanton. Recall from Proposition \ref{pro:caracg2}, that $\lambda=1$ is the only possible value for the connection $\nabla^\lambda$ to be an instanton.

We start the section by giving explicit examples of $\rG_2$-instantons and then show that these are the only ones. We point out that Example \ref{ex:heis} was already introduced as an instanton in \cite{FIUV11}; we include it here for the sake of completeness. 

Below, we will denote $\mh_{2m+1}$ the Heisenberg Lie algebra of dimension $2m+1$, that is, the one having a dual basis $\{e^1, \ldots, e^{2m+1}\}$ satisfying
\begin{eqnarray*}
   \left\{ 
    \begin{array}{l}
    \d e^i=0,\quad i=1, \ldots, m,\\
\dd e^7=e^{12}+e^{34}+\cdots+e^{2m-1\,2m}.
    \end{array}
    \right.
\end{eqnarray*}
Also, we write $\mh_\H$ and $\R\oplus\mn_{3,2}$, respectively, for the quaternionic Heisenberg and the direct sum of the free 2-step nilpotent Lie algebra with an abelian factor, each of which  having a dual basis $\{e^1, \ldots,e^7\}$ such that the differentials satisfy
\begin{eqnarray*}
\mh_\H:  &&
    \d e^i=0,\quad i=1, \ldots, 4,\quad
\dd e^5= e^{13}-e^{24},  \dd e^6=-e^{14}-e^{23}, \,\dd e^7=e^{12}+e^{34},\\
\R\oplus\mn_{3,2} :&&
     \d e^i=0,\quad i=1, \ldots, 4,\quad
 \dd e^5=  e^{23},\,  \dd e^6= e^{24},\, \dd e^7= e^{34}.
     \end{eqnarray*}

\begin{ex}\label{ex:heis} Let $\mg$ be the 7-dimensional 2-step nilpotent Lie algebra with a basis $\{ e^1, \ldots, e^7\}$ of $\mg^*$ with structure equations
\begin{eqnarray}
   \left\{ 
    \begin{array}{l}
    \d e^i=0,\quad i=1, \ldots, 7,\\
\dd e^7=a (e^{12}-e^{56})+b(e^{34}-e^{56}),
    \end{array}
    \right.
  \label{eq:heis} 
\end{eqnarray}
    for some $a,b\in\R$ not simultaneously vanishing. The Lie algebra $\mg$ has commutator of dimension 1. Moreover, depending on the parameters $a,b$, it is isomorphic to either $\mh_7$ or $\R^4\oplus \mh_5$. The parameters $a,b$ correspond to families of metrics on these two Lie algebras. 
   
Let $\varphi$ be the $\rG_2$-structure induced by the above basis, namely, 
\begin{equation}
	 \varphi=e^{127}+e^{347}+e^{567} + e^{135} -e^{146} -e^{236}-e^{245}.
\end{equation}Hence the basis above becomes orthonormal with respect to the induced metric. Canonical computations using \eqref{eq:heis} show that $\varphi$ is purely coclosed. Moreover, the basis and the $\rG_2$-structure satisfy Lemma \ref{lm:can1}. Then, by Theorem \ref{teo:gprim1} the connection $\nabla^1$ defined in \eqref{eq:np} by $\varphi$ with $\lambda=1$ is a $\rG_2$-instanton. 

The torsion $T$ of $\nabla^1$ is given in \eqref{eq:Tdim1}, which becomes
\[
T=(a (e^{12}-e^{56})+b (e^{34}-e^{56}))\wedge e^7=\d e^7\wedge e^7. 
\]Explicitly, the components of the matrix of 1-forms given by the connection form \eqref{eq:gprim1psi} are
\[
(\Psi_1)^1_2=ae^7,\quad (\Psi_1)^3_4=be^7,\quad (\Psi_1)^5_6=-(a+b)e^7.
\] Notice that $\nabla_u^1=0$ for all $u\bot e_7$. Moreover, since $\Psi_1\wedge \Psi_1=0$, the curvature of $\nabla^1$ is $R^1=\d \Psi_1$, so the non-zero curvature components are (see \eqref{eq:gprim1R})
\begin{eqnarray*}
   & (R^1)^1_2=a (a e^{12}+be^{34}-(a+b)e^{56}),\quad (R^1)^3_4=b(a e^{12}+be^{34}-(a+b)e^{56}),\\ &(R^1)^5_6=-(a+b)(a e^{12}+be^{34}-(a+b)e^{56}).
\end{eqnarray*}Proceeding as in Theorem \ref{teo:gprim1}, one can easily check that $R^1\wedge \psi=0$. Also, from the last expression, we obtain that the curvature tensor satisfies $R^1(x,y,z,w)=R^1(z,w,x,y)$  for every $x,y,z,w\in\mg$. In addition, the holonomy of this connection 
$\mathfrak{hol}(\nabla^1)=\{(R^1)^i_j:i,j=1, \ldots, 7\}$ is 1-dimensional and contained in $\su(3)\subset\so(\mg)$.
\end{ex}

\begin{ex}\label{ex:qheis} Consider the 2-step nilpotent Lie algebra $\mg$ with a basis $\{e^1,\ldots, e^7\}$ of $\mg^*$ for which the structure equations  are 
\begin{eqnarray}
   \left\{ 
    \begin{array}{l}
    \d e^i=0,\quad i=1, \ldots, 4,\\
\dd e^5=\nu (e^{13}-e^{24}), \, \dd e^6=\nu(-e^{14}-e^{23}), \,\dd e^7=\nu(e^{12}+e^{34}),
    \end{array}
    \right.
  \label{eq:qheis} 
\end{eqnarray}
for some $\nu\neq 0$. 
One should notice that for any $\nu\neq 0$, the Lie algebra with structure constants \eqref{eq:qheis} is isomorphic to the quaternionic Heisenberg $\mh_\H$. The parameter $\nu$  corresponds to a 1-parameter family of metrics on $\mh_\H$.

Let $\varphi$ be the $\rG_2$-structure \eqref{eq:gcvar} induced by the above basis so that $\{e^1, \ldots, e^n\}$ becomes an orthonormal basis. Also, if  $\{\sigma_i^+\}_{i=1}^3$ denotes the basis self-dual forms \eqref{eq:sigmabasis} in $\Lambda^2\langle \{e_1, \ldots,e_4\}\rangle^*$, $\varphi$ has the form \eqref{eq:phibody} and thus its Hodge dual is \eqref{eq:stphi}. Using these expressions together with  \eqref{eq:n32}, one can easily show that $\varphi$ is coclosed. 

We claim that the connection $\nabla^1$ defined by $\varphi$ via \eqref{eq:np}  is a $\rG_2$-instanton for  $\lambda=1$. For this, we will show that the conditions in  Proposition \ref{pro:caracg2} are satisfied.

It is straightforward that $\mg'$ is 3-dimensional and spanned by $e_5,e_6,e_7$, thus $\varphi$ calibrates $\mg'$. Moreover,  it is clear from \eqref{eq:qheis} that 
\[
\d e^{i+4}=\nu\sigma_i^+=(\d e^i)^+, \qquad i=1,2,3.
\] So (2) in Proposition \ref{pro:caracg2} holds for $\mu=3\nu\neq 0$. Finally, using \eqref{eq:jc}, \eqref{eq:qheis} and the fact that $\{e_1, \ldots, e_7\}$ is an orthonormal basis, one readily computes
\[
j(e_5)=\nu
    \left(
    \begin{smallmatrix}
    0&0&1&0\\
    0&0&0&-1\\
    -1&0&0&0\\
    0&1&0&0
    \end{smallmatrix}\right),\quad
j(e_6)=\nu
    \left(
    \begin{smallmatrix}
    0&0&0&-1\\
    0&0&-1&0\\
    0&1&0&0\\
    1&0&0&0
    \end{smallmatrix}\right),\quad
j(e_7)=\nu
    \left(
    \begin{smallmatrix}
    0&1&0&0\\
    -1&0&0&0\\
    0&0&0&1\\
    0&0&-1&0
    \end{smallmatrix}\right),
\]from which it is easy to verify that (3) in Proposition \ref{pro:caracg2} holds.  Therefore, $\nabla^1$ is a $\rG_2$-instanton as claimed.

For the sake of completeness, we give explicit expressions for the torsion, the connection $\nabla^1$ and its curvature, which can be obtained from the formulas in Section \ref{sec:appadapt}, since Assumption \ref{as0} holds. The torsion $T$ defined by $\varphi$ via \eqref{eq:tor} is given in \eqref{eq:tor2}. Explicitly,
\[
T=\nu(-4 e^{567}+ \sigma_1^+\wedge e^5+\sigma_2^+\wedge e^6+ \sigma_3^+\wedge e^7)=-4\nu e^{567}+ \sum_{i=5}^7\d e^i\wedge e^i.
\] Moreover, by Lemma \ref{lm:npend}, the connection $\nabla^1$ verifies  $\nabla^1_u=0$ for all $u\in\mq$ and $\nabla_z^1=-j(z)-2\nu\tau(z)$, for all $z\in\mg'$, where $\tau(z)\in\so(\mg')$ is the skew-symmetric map corresponding to $z\lrcorner (\varphi|_{\mg'})$. Hence, the connection 1-form  of $\nabla^1$ as presented in \eqref{eq:psilam} is
{\small\[
\Psi_1=\nu
\left(
\begin{array}{c|c}
\begin{array}{cccc}
   0&-e^7&-e^5&e^6\\
    e^7&0&e^6&e^5\\
    e^5&-e^6&0&-e^7\\
    -e^6&-e^5&e^7&0\\
\end{array}&0\\
\hline
0&\begin{array}{ccc}
 0&2e^7&-2e^6\\
 -2e^7&0&2e^5\\
 2e^6&-2e^5&0
\end{array}
\end{array}
\right)
\]}
and the non-zero curvature components of the curvature tensor $R^1=\d \Psi_1+\Psi_1\wedge \Psi_1$:
\begin{align*}
	(R^1)^1_{2}=(R^1)^3_{4}=-\frac{1}{2}(R^1)^5_{6}=\nu^2\left(-e^{12}-e^{34}+2e^{56}\right),\\
	-(R^1)^1_{3}=(R^1)^2_{4}=\frac{1}{2}(R^1)^6_{7}=\nu^2\left(e^{13}-e^{24}-2e^{67}\right),\\
	(R^1)^1_{4}=(R^1)^2_{3}=-\frac{1}{2}(R^1)^5_{7}=\nu^2\left(-e^{14}-e^{23}+2e^{57}\right).
\end{align*}
The Hodge dual $\psi$ of $\varphi$ satisfies \eqref{eq:stphi}, so one can easily check that $R^1\wedge \psi=0$. 
Also, from this expression we obtain that the curvature tensor satisfies $R^1(x,y,z,w)=R^1(z,w,x,y)$  for every $x,y,z,w\in\mg$. 
In addition, one can check that   $\mathfrak{hol}(\nabla^1)\simeq\su(2)$.
\end{ex}

\begin{ex}\label{ex:n32} Let $\mg$ be the 2-step nilpotent Lie algebra with basis $\{e^1,\ldots, e^7\}$ of $\mg^*$ for which the structure equations  are 
\begin{eqnarray}
   \left\{ 
    \begin{array}{l}
    \d e^i=0,\quad i=1, \ldots, 4,\\
 \dd e^5= -2\nu e^{24},\,  \dd e^6= -2\nu e^{23},\, \dd e^7= 2\nu e^{34},
    \end{array}
    \right.
  \label{eq:n32} 
\end{eqnarray}
for some $\nu\neq 0$. 
For every $\nu$, the Lie algebra $\mg$ is isomorphic to $\R\oplus\mn_{3,2}$, where $\mn_{3,2}$ is the free 2-step nilpotent Lie algebra on 3-generators. As before, the parameter $\nu$  corresponds to a 1-parameter family of metrics on $\R\oplus\mn_{3,2}$.

Let $\varphi$ be the $\rG_2$-structure \eqref{eq:gcvar} defined by the basis $\{e^1, \ldots, e^7\}$, which thus becomes orthonormal with respect to the induced metric. Clearly, $\varphi$ has the form \eqref{eq:phibody}, where $\sigma_i^+$ are as in \eqref{eq:sigmabasis}, and its Hodge dual is \eqref{eq:stphi}. 
Using these expressions together with  \eqref{eq:n32}, one can easily show that $\varphi$ is coclosed. The Lie algebra $\mg$, with the $\rG_2$-structure, verifies the conditions in  Proposition \ref{pro:caracg2} as we show next.

The commutator $\mg'$ is spanned by $e_5,e_6,e_7$ and thus $\varphi$ calibrates $\mg'$. Notice that, contrary to Example \ref{ex:qheis}, $\dim\mz=4$ since $\mz=\langle e_1\rangle\oplus \mg'$. Let $\sigma_i^+$ be the basis of self-dual forms in \eqref{eq:sigmabasis} and consider the next basis of anti-self-dual forms:
\begin{equation*}
	\sigma_1^-=e^{13}+e^{24}, \quad \sigma_2^-=e^{14}-e^{23},\quad \sigma_3^-=e^{12}-e^{34}.
\end{equation*} 
Then,  by \eqref{eq:n32} we have
\[
\d e^{5}=\nu(\sigma_1^+-\sigma_1^-),\quad \d e^{6}=\nu(\sigma_2^++\sigma_2^-)\quad \d e^{7}=\nu(\sigma_3^+-\sigma_3^-),
\]which implies, 
\[
(\d e^{i+4})^+=\nu\sigma_i^+, \qquad i=1,2,3,
\] so (2) in Proposition \ref{pro:caracg2} holds for $\mu=3\nu\neq 0$. Using \eqref{eq:n32} and the orthonormal basis $\{e_1, \ldots, e_7\}$, we get that the matrices of the $j(z)$ map in \eqref{eq:jc} are
\[
j(e_5)=2\nu
    \left(
    \begin{smallmatrix}
    0&0&0&0\\
    0&0&0&-1\\
    0&0&0&0\\
    0&1&0&0
    \end{smallmatrix}\right),\quad
j(e_6)=2\nu
    \left(
    \begin{smallmatrix}
    0&0&0&0\\
    0&0&-1&0\\
    0&1&0&0\\
    0&0&0&0
    \end{smallmatrix}\right),\quad
j(e_7)=2\nu
    \left(
    \begin{smallmatrix}
    0&0&0&0\\
    0&0&0&0\\
    0&0&0&1\\
    0&0&-1&0
    \end{smallmatrix}\right).
\]Notice that $j(z)e_1=0$ for all $z\in\mg'$, because $e_1\in\mz\cap(\mg')^\bot$. With the above expressions, it is easy to verify that (3) in Proposition \ref{pro:caracg2} holds.  Therefore, $\nabla^1$ is indeed a $\rG_2$-instanton.

As in Example \ref{eq:qheis} we can give explicitly the torsion, the connection and curvature formulae by using the formulas in Section \ref{sec:appadapt}. The torsion $T$ defined is the one given in \eqref{eq:tor2}, namely,
\[
T=2\nu(-2e^{567}- e^{245}-e^{236}+e^{347})=-4\nu e^{567}+ \sum_{i=5}^7\d e^i\wedge e^i.
\] Also, by Lemma \ref{lm:npend}, $\nabla^1_u=0$ for all $u\in\mq$ and $\nabla_z^1=-j(z)-2\nu\tau(z)$, for all $z\in\mg'$. Hence, the connection 1-form  $\Psi_1$ of $\nabla^1$ given in \eqref{eq:psilam} is
{\small\[
\Psi_1=2\nu
\left(
\begin{array}{c|c|c}
0& 0&0\\
\hline
0& \begin{array}{ccc} 
 0&e^6&e^5\\
 -e^6&0&-e^7\\
 -e^5&e^7&0
\end{array}
&0\\
\hline
0&0&\begin{array}{ccc}
     0&e^7&-e^6\\
    -e^7&0&e^5\\
    e^6&-e^5&0
\end{array}
\end{array}
\right)
\]}
and the non-zero curvature components of the curvature tensor $R^1=\d \Psi_1+\Psi_1\wedge \Psi_1$:
\begin{align*}
	(R^1)^2_{3}=-(R^1)^5_{7}=4\nu^2(-e^{23}+e^{57}),\\
 (R^1)^2_{4}=(R^1)^6_{7}=4\nu^2(-e^{24}-e^{67}),\\
 (R^1)^3_{4}=-(R^1)^5_{6}=4\nu^2(-e^{34}+e^{56}).
\end{align*}
Since $\psi$ satisfies \eqref{eq:stphi},  one can easily check that $R^1\wedge \psi=0$. Also, the curvature of $\nabla^1$ satisfies the symmetry $R^1(x,y,z,w)=R^1(z,w,x,y)$  for every $x,y,z,w\in\mg$. Moreover, it follows that  $\mathfrak{hol}(\nabla^1)\simeq\so(3)$.
\end{ex}

\begin{remark}\label{rem:nonis}
Within each family of Lie algebras with $\rG_2$-structures appearing in Examples \ref{ex:heis}, \ref{ex:qheis} and \ref{ex:n32}, infinitely many are non-equivalent in the following sense. 

If there exists a Lie algebra automorphism between two 2-step nilpotent Lie algebras with $\rG_2$-structures, $f:(\mg_1,\varphi_1)\to(\mg_2,\varphi_2) $, then $f(\mg_1')=\mg_2'$. Moreover, $f$ is an isometry of the metrics $g_i$, $i=1,2$, induced on $\mg_i$. Therefore, $f((\mg_1')^\bot)=(\mg_2')^\bot$ and one can easily show that
\[
f^{-1}j(f(z))f=j(z), \qquad \forall z\in\mg_1',
\]which implies that the symmetric endomorphisms $j(z)^2$ and $j(f(z))^2$ have the same eigenvalues. 

On the previous examples, the parameters $a^2,b^2, (a+b)^2$ or $-\nu^2$ are corresponding eigenvalues of $j(z)^2$ for unit length vectors $z\in\mg'$. So when these values are different for two Lie algebras in the same family of each example, there is no Lie algebra automorphism between them preserving the $\rG_2$-structure.
\end{remark}

\begin{teo} \label{teo:classif}
On a $2$-step nilpotent Lie algebra $\mg$, there exists a coclosed $\rG_2$-structure $\varphi$ for which $\nabla^1$ is a $\rG_2$-instanton if and only if  there exists a basis $\{e^1, \ldots, e^7\}$ of $\mg^*$ such that $\varphi$ has the form \eqref{eq:gcvar} and the structure constants in that basis satisfy one of the following conditions:
\begin{enumerate}
    \item $\d e^i=0$ for $i=1, \ldots, 6$ and $\d e^7=a (e^{12}-e^{56})+b(e^{34}-e^{56})$, for some $a,b$ non simultaneously zero.
    \item $\d e^i=0$ for $i=1, \ldots, 4$ and $\dd e^5=\nu (e^{13}-e^{24})$, $\dd e^6=\nu(-e^{14}-e^{23})$, $\dd e^7=\nu(e^{12}+e^{34})$, for some $\nu\neq 0$.
     \item $\d e^i=0$ for $i=1, \ldots, 4$ and $\dd e^5=-2\nu e^{24}$, $\dd e^6=-2\nu e^{23}$, $\dd e^7=2\nu e^{34}$, for some $\nu\neq 0$.
\end{enumerate}
\end{teo}

\begin{proof}
Suppose that $\nabla^1$ is a $\rG_2$-instanton for some coclosed $\rG_2$-structure $\varphi$ on $\mg$. 

If $\dim \mg'=1$, then by Lemma \ref{lm:can1} there exists an orthonormal basis $\{e^1, \ldots, e^7\}$ of $\mg^*$ for which $\varphi$ has the desired form and the structure constant verify \eqref{eq: structure_equation_dim_1}. However, since $\nabla^1$ is an instanton, Theorem \ref{teo:gprim1} implies that  $c=-(a+b)$. Hence $\mg$ satisfies (1) in the statement.

Now assume that $\dim\mg'\geq 2$, then by Proposition \ref{pro:caracg2}, we actually have $\dim \mg'=3$, and also $\varphi$ calibrates $\mg'$. Moreover, (2) in Proposition \ref{pro:caracg2} implies that there exists an orthonormal basis $\{e_1, \ldots, e_7\}$ of $\mg$ such that $\mg'$ is spanned by $\{e_5,e_6,e_7\}$, $\varphi$ satisfies \eqref{eq:gcvar} and 
\begin{equation}\label{eq:deimu}
(\d e^i)^+=\frac{\mu}3\sigma_i^+,
\end{equation}for some $\mu\neq 0$, where $\sigma_i^+$ is the basis in \eqref{eq:sigmabasis}.
In particular, Assumption \ref{as0} holds and the results in Section \ref{sec:appadapt} apply to $\varphi$.

If $\d e^i=(\d e^i)^+$ for $i=5,6,7$, it is straightforward that (2) in the statement holds for $\nu=\frac{\mu}3$. 
Assume then that there is some $i\in\{5,6,7\}$ for which $\d e^i\neq (\d e^i)^+$, that is, $(\d e^i)^-\neq 0$.  By Corollary \ref{cor:comker}, there exists $0\neq x\in\mr$ such that $j(z)x=0$ for all $z\in\mg'$. The subgroup of $\rG_2$ fixing the calibrated subspace $\mg'$ acts like $\SO(4)$ in $(\mg')^\bot$, so we may assume that $x=e_1$ and  \eqref{eq:deimu} still holds. Now, $j(z)e_1=0$ implies $e_1\lrcorner \d e^i=0$ which together with \eqref{eq:deimu} implies 
\[
\d e^5=-2\frac{\mu}3 e^{24},\quad \d e^6=-2\frac{\mu}3e^{23},\quad  \d e^7=2\frac{\mu}3 e^{34}.
\]Hence (3) holds taking $\nu=\frac{\mu}3$.

The converse follows from Examples \ref{ex:heis}, \ref{ex:qheis} and \ref{ex:n32}.
\end{proof}

By looking at the isomorphism classes of Lie algebras satisfying one of the conditions in Theorem \ref{teo:classif}, we get the following.
\begin{cor}
    A $2$-step nilpotent Lie algebra $\mg$
 of dimension $7$ admits a coclosed $\rG_2$-structure
such that $\nabla^1$ is a $\rG_2$ instanton if and only if $\mg$ is isomorphic to either $\R^2\oplus \mh_{5}$, $\mh_{7}$, $\mh_\H$ or $\R\oplus\mn_{3,2}$.
\end{cor}

A simply connected complete Riemannian manifold is naturally reductive if there exists a metric connection $\nabla^c$ such that the Riemann curvature tensor and the torsion of $\nabla^c$ are parallel with respect to $\nabla^c$ \cite[Theorem 2.3]{TrVa84}. We will show next that the Riemannian manifolds arising from the $\rG_2$-instantons on the 2-step nilpotent Lie groups are naturally reductive with connection $\nabla^c=\nabla^1$.

\begin{teo}Let $G$ be a simply connected $2$-step nilpotent Lie group of dimension $7$, let $\varphi$ be a left-invariant coclosed $\rG_2$-structure and denote by $g$ and $R$, respectively, the left-invariant metric induced by $\varphi$ and its Riemann curvature tensor. If the connection $\nabla^1$ defined by $\varphi$ is a $\rG_2$-instanton, then $\nabla^1T=\nabla^1 R=0$ and therefore $(G,g)$ is naturally reductive. 
\end{teo}
\begin{proof}By Theorem \ref{teo:classif}, if $\nabla^1$ is an instanton, then the Lie algebra $\mg$ of $G$ has a dual basis $\{e^1, \ldots, e^7\}$ such that the $\rG_2$-structure is given by \eqref{eq:gcvar} and the differentials verifiy the conditions of one of the Examples \ref{ex:heis}, \ref{ex:qheis} or \ref{ex:n32}, or respectively, (1), (2), or (3) in Theorem \ref{teo:classif}. In particular, the metric $g$ induced by $\varphi$ is the one making this basis orthonormal. In each of these examples, it was shown that the curvature $R^1$ has the symmetry $R^1(x,y,z,w)=R^1(z,w,x,y)$ for all $x,y,z,w\in\mg$. Then, by \cite[Corollary 3.4]{Iva02}, $\nabla^1 T$ is a 4-form. 

However, it was noted in the examples that $\nabla_u^1=0$ for all $u\in \mr$. In addition, if $\dim \mg'=1$ we have $T\in \Lambda^2\mr^*\otimes \Lambda^1(\mg')^*$, while $T\in( \Lambda^2\mr^*\otimes \Lambda^1(\mg')^*)\oplus\Lambda^3(\mg')^*$ when $\dim \mg'=3$. Hence the $4$-form $\nabla^1T$ vanishes. It remains to show that $\nabla^1 R=0$, which is equivalent to prove that $\nabla^1_z R=0$ for all $z\in\mg'$, since $\nabla^1_u=0$ for all $u\in(\mg')^\bot$. 

In case (1) of Theorem \ref{teo:classif}, it is easy to check that for any $z\in\mg'$,  $\nabla^1_z=-j(z)\in \so(\mr)$. Canonical computations, using the fact that $\dim(\mg')=1$, show that the extension by zero of $j(z)$ to the whole $\mg$ is a skew-symmetric derivation of $(\mg,g)$. Therefore, $e^{tj(z)}$ is an isometry of $(G,g)$  and thus $(e^{tj(z)})^*R=R$ for all $t\in\R$. Differentiating this equation, one gets that the action of $j(z)$ on $R$ vanishes, that is $j(z)R=-\nabla_z^1R=0$.

In cases (2) and (3) of Theorem \ref{teo:classif}, Lemma \ref{lm:npend} implies that for any $z\in\mg'$, $\nabla^1_z=-j(z)-2\nu\tau(z)$, where $\tau(z)\in\so(\mg')$ is the endomorphism corresponding to $z\lrcorner (\varphi|_{\mg'})$ (see also the computations in Examples \ref{ex:qheis} or \ref{ex:n32}).
In addition, by (3) in Proposition \ref{pro:caracg2}, 
\begin{equation}
[j(z),j(z')]=2\nu j(\tau(z)z'), \qquad \forall z,z'\in\mg'.
\end{equation}
Using this equality, one can show that for a fixed $z\in\mg'$, the skew-symmetric endomorphism  $D:=\nabla_z^1$ is a skew-symmetric derivation of $(\mg,g)$. This implies that $e^{tD}$ is an isometry of $(G,g)$ for all $t$, and thus $\nabla_z^1R=0$ as before.

Therefore $\nabla^1T=\nabla^1R=0$, so $(G,g)$ is naturally reductive and $\nabla^1=\nabla^c$ by \cite[Theorem 2.3]{TrVa84}.
\end{proof}

\begin{remarks}\begin{enumerate}
\item The converse of the above theorem does not hold. In fact, $\mh_3\oplus\R^4$ admits naturally reductive metrics \cite{TrVa83}. However, it does not admit purely-coclosed $\rG_2$-structures \cite{dBMR} and thus $\nabla^1$ does not define a $\rG_2$-instanton for any coclosed $\rG_2$-structure, due to Theorem \ref{teo:gprim1}.

\item For a general Riemannian manifold, the condition for it to be naturally reductive requires the specification of a transitive subgroup of isometries $H$. In the case of nilpotent Lie groups endowed with left-invariant metrics, there is no such a need because $(G,g)$ is naturally with respect to a subgroup $H$ if and only if it is naturally reductive with respect to its full isometry group $\Iso(G,g)$ \cite[Theorem 3.5]{LA3}.

\item  The fact that $(G,g)$ is naturally reductive when $\nabla^1$ is an instanton can also be shown by a result of Gordon \cite{GO2}. In fact, she proves that $(G,g)$ is naturally reductive if and only if the the image of the map $j: \mg'\to\so(\mr)$ is a subalgebra of $\so(\mr)$ and also,  for each $z\in \mg'$  the map $z'\mapsto j^{-1}[j(z),j(z')]$ is in $\so(\mg')$. One can easily check that the latter conditions hold when $\nabla^1$ is a $\rG_2$-instanton. Indeed, for $\dim\mg'=1$ they are trivially satisfied, and for $\dim\mg'\geq 2$ they follow from \eqref{eq:jcondteo}. 

\item The additional symmetry of the curvature tensor $R^1$ when $\nabla^1$ is an instanton, namely \linebreak $R^1(x,y,z,w)=R^1(z,w,x,y)$ for all $x,y,z,w\in\mg$, implies that $T$ is a Killing 3-form on $\mg$ \cite[Corollary 3.4]{Iva02}, that is, $\nabla^gT=\frac14 \d T$. It was shown in \cite{dBM2} that every left-invariant Killing 3-form on a 2-step nilpotent Lie group endowed with a left-invariant metric $(G,g)$ is a multiple of the torsion form of the naturally reductive connection $\nabla^c$.

\item In \cite{AFS15}, the authors study naturally reductive structures on $\mh_\H$. They show that the naturally reductive connection coincides with $\nabla^1$ for some coclosed $\rG_2$-structure, and therefore one obtains $\nabla^1 T=\nabla^1 R=0$.

\item A $\rG_2$-instanton is constructed on $\R\oplus\mn_{3,2}$, with the metric corresponding to $\nu=1$ in \eqref{eq:n32}, by lifting an $\SU(3)$-instanton in \cite[\S 6.2]{IvIv05}. The resulting $\rG_2$-structure is not coclosed since its torsion form $\tau_1$ is non-zero.
    \end{enumerate}
\end{remarks}

\section{\texorpdfstring{$\rG_2$}{G2}-structures Calibrating a Central Subspace Containing the Center}\label{sec:appadapt}

Along this section, $\mg$ is a 2-step nilpotent Lie algebra of dimension 7 with $\dim\mg'\geq 2$, and $\varphi$ is a coclosed $\rG_2$-structure satisfying the following:

\begin{assumption} \label{as0} There exist a  subspace $\mc$ of $\mg$, verifying $\mg'\subset\mc\subset\mz$ and calibrated by $\varphi$.
\end{assumption}

The aim of this section is to make explicit computations for the connection $\nabla^\lambda$ in  \eqref{eq:np} defined by the coclosed $\rG_2$-structures under this assumption. In particular, we will compute its curvature components and study the instanton condition on it.

As pointed out at the beggining of Subsection \ref{ss:nsc}, under this assumption, there exists an orthonormal basis  $\{e_1, \ldots, e_7\}$ of $\mg$ such that $\mc=\langle\{e_5,e_6,e_7\}\rangle$ and \eqref{eq:gcvar} holds. Moreover, if $\mq$ denotes the orthogonal of $\mc$, then $\alpha_i:=\d e^i\in \Lambda^2\mq^*$ for $i=5,6,7$, whilst $\d e^i=0$ for $i=1, \ldots, 4$ (see \eqref{eq:deial}).

Consider the orientation $e^{1234}$ on $\mc^\bot$ and the splitting of $\Lambda^2\mq^*$ into self-dual and anti-self dual forms, so that $\alpha_i=\alpha_i^++\alpha_i^-$ as in \eqref{eq:strbody}. In addition, let $\{\sigma_1^+,\sigma_2^+,\sigma_3^+\}$ be the basis in \eqref{eq:sigmabasis} so that $\varphi$ takes the form \eqref{eq:phibody}.

In order to achieve manageable expressions for the connection under study and its curvature, we start by introducing the following.
 
\subsection{Notation.}\label{ss:not}
Let $[\,,\,]$ denote the Lie bracket of $\mg$. We set other three Lie algebra structures on $\mg$, whose Lie brackets will be denoted by $[\,,\,]_+$, $[\,,\,]_-$ and $[\,,\,]_0$ as follows
\begin{eqnarray}
\label{eq:brbeta}
 &\,[u,v]_\pm=-\sum_{i=5}^7 \alpha_i^\pm(u,v)e_i,\quad \,[u,v]_0=\sum_{i=5}^7 \sigma_i^+(u,v)e_i,\quad \forall\, u,v\in\mq, \\
 &\mbox{ and }\quad [x,z]_\pm=[x,z]_0=0, \quad \forall x\in\mg, \,z\in\mc.
\end{eqnarray}
It is easy to check that each $(\mg,[\cdot,\cdot]_{\pm,0})$ is either an abelian or a $2$-step nilpotent Lie algebra, its commutator is contained in $\mc$, and $\mc$ is contained in the center.

For each $z\in \mc$,  let $j(z)^{\pm,0}$ denote the skew-symmetric map defined on $\mq$ by the equations
\begin{equation}
\label{eq:j0}
\begin{array}{rcl}
\lela j(z)^{\pm} u,v\rira&=&\lela[u,v]_{\pm},z\rira=-\sum_{i=5}^7 \alpha_i^\pm(u,v)g(z,e_i),\qquad \forall u,v\in\mq,\\
\\
\lela j(z)^{0} u,v\rira&=&\lela[u,v]_{0},z\rira=\sum_{i=5}^7 \sigma_i^+(u,v)g(z,e_i),\qquad \forall u,v\in\mq.
\end{array}
\end{equation}
In particular, for every  $i=5,6,7$,  $\lela j(e_i)^\pm u,v\rira=-\alpha_i^\pm(u,v)$  and $g(j(e_i)^0,u,v)=\sigma_{i-4}^+(u,v)$, for all $u,v\in\mq$.

It is straightforward that if $j(z)$ is the map corresponding to $[\,,\,]$ via \eqref{eq:jc}, then 
\[
j(z)^++j(z)^-=j(z), \qquad \forall z\in\mc.
\]

Under the identification of $\Lambda^2\mq^*$ with $\so(\mq)$ via the metric, it is easy to check that the space of self-dual and anti-self dual forms on $\Lambda^2\mq^*$ induces a splitting as a direct sum of (commuting) ideals $\so(\mq)=\so(\mq)^+\oplus\so(\mq)^-$, where each ideal is isomorphic to $\so(3)$. Through this identification, for any $z\in\mc$, $\d z^\flat$ corresponds to $-j(z)$, and its components $(\d z^\flat)^\pm$ correspond to $-j(z)^\pm$; that is, $j(z)^\pm\in\so(\mq)^\pm$. Moreover,  $j(z)^0\in\so(\mq)^+$ as well and one has
\begin{equation}\label{eq:jzconmut}
[j(z)^0,j(z')^-]=[j(z)^+, j(z')^-]=0,\qquad \mbox{  for every }z,z'\in\mc.
\end{equation}

\subsection{The Linear Connection \texorpdfstring{$\nabla^\lambda$}{D} and its Curvature \texorpdfstring{$R^\lambda$}{R}.}\label{ss:con} 
Recall that under Assumption \ref{as0},  $\varphi$ has the form \eqref{eq:phibody} and its Hodge dual is given in \eqref{eq:stphi}. Let us denote $\mu:=\tr S$ where $S=(s_{ij})_{i,j=1}^3$ is defined in \eqref{eq:cc}.  By \eqref{eq:strbody} and \eqref{eq:phibody}  we get
\[
\d\varphi=2\mu e^{1234}+(\alpha_5^++\alpha_5^-)\wedge e^{67}-(\alpha_6^++\alpha_6^-)\wedge e^{57}+(\alpha_7^++\alpha_7^-)\wedge e^{56},
\]
which implies
\begin{equation}\label{eq:stdp}
-\star \d\varphi=-2\mu e^{567}-(\alpha_5^+-\alpha_5^-)\wedge e^{5}-(\alpha_6^+-\alpha_6^-)\wedge e^{6}-(\alpha_7^+-\alpha_7^-)\wedge e^{7}.
\end{equation}
In addition, 
\eqref{eq:pwdp} yields
\begin{equation}
\label{eq:t0}
\star(\varphi\wedge \d\varphi)=4\mu.
\end{equation}
Therefore, the $\rG_2$-structure $\varphi$ is purely-coclosed if and only if  $\mu=\tr S=0$.

Since $\varphi$ calibrates $\mc$,  for any $z\in\mc$, there is a skew-symmetric map  $\tau(z)\in\so(\mc)$ corresponding to $z\lrcorner (\varphi|_{\mg'})$, namely, 
\begin{equation}
\label{eq:taudef}
g(\tau(z)z',z'')=g(z\lrcorner(\varphi|_{\mc})z',z'')=\varphi(z,z',z''), \qquad \forall z',z''\in\mc.
\end{equation}
It is straightforward from  \eqref{eq:stdp}  and \eqref{eq:t0} above, that the torsion \eqref{eq:tor} is 
\begin{eqnarray}
    \label{eq:tor2}
T&=&\frac23\mu\varphi -2\mu e^{567}-(\alpha_5^+-\alpha_5^-)\wedge e^{5}-(\alpha_6^+-\alpha_6^-)\wedge e^{6}-(\alpha_7^+-\alpha_7^-)\wedge e^{7}\\ \nonumber
&=&-\frac43\mu e^{567}
+\sum_{i=5}^7(-\alpha_i^++\alpha_i^-+\frac23\mu\sigma_{i-4}^+)\wedge e^{i}.
\end{eqnarray}
Note that $T\in( \Lambda^2\mq^*\otimes \mc^*)\oplus \Lambda^3\mc^*$. For each $\lambda\in\R$, let  $\nabla^\lambda$ be the metric connection defined in \eqref{eq:np}, namely, 
\begin{equation}
\label{eq:npapp}
\lela \nabla_x^\lambda y,z\rira =\lela \nabla_x^gy,z\rira +\frac{\lambda}2 T(x,y,z), \qquad x,y,z\in\mg,
\end{equation}where $T$ is given in \eqref{eq:tor2}.

We use the notation in the previous subsection to give explicit expressions for the connection $\nabla^\lambda$.
\begin{lm}\label{lm:npend}
For any vector $x\in \mg$, the endomorphism $\nabla^\lambda_x\in \so(\mg)$ given in \eqref{eq:npapp} is:
\begin{eqnarray}
\label{eq:npend1}
\nabla_u^\lambda&=&\frac{(\lambda+1)}2(\ad_u^+-(\ad_u^+)^*)-\frac{(\lambda-1)}2(\ad_u^--(\ad_u^-)^*)+\frac13\lambda\mu (\ad_u^0-(\ad_u^0)^*),\quad\forall u\in\mq, \\
\nabla_z^\lambda&=&\frac{(\lambda-1)}2j(z)^+-\frac{(\lambda+1)}2j(z)^-+\frac13\lambda\mu \,j(z)^0-\frac23 \mu\lambda \;\tau(z), \quad \forall z\in\mc.\label{eq:npend2}
\end{eqnarray}
Equivalently, we have
\begin{equation}\label{eq:nablap}
\left\{
\begin{array}{ll}
\nabla^\lambda_u v=\frac{(\lambda+1)}2[u,v]^+-\frac{(\lambda-1)}2[u,v]^-+\frac13\lambda \mu [u,v]^0&\mbox{ if } u,v\in\mq,\\
\nabla^\lambda_u z= -\frac{(\lambda+1)}2(\ad_u^+)^*z+\frac{(\lambda-1)}2(\ad_u^-)^*z-\frac13\mu\lambda(\ad_u^0)^*z&\mbox{ if } u\in\mq,\,z\in\mc,\\
\nabla^\lambda_z u= \frac{(\lambda-1)}2(\ad_u^+)^*z-\frac{(\lambda+1)}2(\ad_u^-)^*z+\frac13\mu\lambda(\ad_u^0)^*z
&\mbox{ if } u\in\mq,\,z\in\mc,\\
\nabla^\lambda_z z'=-\frac23\mu\lambda \; \tau(z)z'& \mbox{ if } z, z'\in\mc.
\end{array}\right.
\end{equation}
\end{lm}
\begin{proof}
Consider the orthogonal decomposition $\mg=\mq\oplus\mc$. For any $u,v,w\in \mq$, the right hand side of \eqref{eq:npapp} is zero because of \eqref{eq:nabla} and \eqref{eq:tor2}, implying that $\nabla_u^\lambda v\in\mc$. In addition, for any $u,v\in \mq$, $z\in\mc$, by \eqref{eq:tor2}, we have
\begin{align*}
2\lela \nabla_u^\lambda v,z\rira
=&\left(\sum_{i=5}^7\left(-(\lambda+1)\alpha_i^++(\lambda-1)\alpha_i^-+\frac23\lambda\mu\sigma_{i-4}^+\right)\wedge e^i\right)(u,v,z).
\end{align*}
In addition, for any $z,z'\in\mc$, $u\in\mq$, 
\[
\lela \nabla_u^\lambda z,z'\rira =\lela \nabla_u^gz,z'\rira +\frac{\lambda}2 T(u,z,z')=0=\lela \nabla_z^\lambda u,z'\rira , 
\]and thus $\nabla_u^\lambda z, \nabla_z^\lambda u\in \mq$. Similar computations give
\begin{eqnarray*}
\lela \nabla_u^\lambda z,v\rira &=&-\lela z,\nabla_u^gv\rira-\frac{\lambda} 2 T(u,v,z)=- \lela z,\nabla_u^\lambda v\rira 
\end{eqnarray*}
and furthermore, since $\nabla^g_uz=\nabla^g_zu$, we have
\begin{align*}
2\lela \nabla_z^\lambda u,v\rira =&2\lela \nabla_u^gz,v\rira+\lambda T(z,u,v)=-2\lela \nabla_u^gv,z\rira+\lambda T(u,v,z)\\
=&\left(\sum_{i=5}^7\left(-(\lambda-1)\alpha_i^++(\lambda+1)\alpha_i^-+\frac23\lambda\mu\sigma_{i-4}^+\right)\wedge e^i\right)(u,v,z).
\end{align*}
Finally, it is easy to see from \eqref{eq:nabla},  \eqref{eq:tor2} and \eqref{eq:npapp} that $\lela \nabla^\lambda_zz',u\rira=0$ for all $z,z'\in\mc$ and $u\in\mq$, and also,
\begin{equation}\label{eq:nlzzz}
\lela \nabla^\lambda_zz',z''\rira=-\frac23\lambda\mu \varphi(z,z',z'')=g(\tau(z)z',z''),\qquad  \forall \,z,z',z''\in\mc,
\end{equation} where $\tau$ is as in \eqref{eq:taudef}.
\end{proof}

Recall from Proposition \ref{pro:Sdiag}, that if $\varphi$ satisfies Assumption \ref{as0}, then there exists a basis $\{e_1, \ldots, e_7\}$ such that $\mc=\langle\{e_5,e_6,e_7\}\rangle$, $\varphi$ is as in \eqref{eq:phibody} and \[
(\d e^{i+4} )^+=d_{i+4} \sigma_i^+, \qquad i=1,2,3,
\] for some $d_i\in\R$.
This implies that, in this basis,  the symmetric matrix $S$ is diagonal of the form $S={\rm diag}(d_5,d_6,d_7)$ and $\mu=d_5+d_6+d_7$. In the notation of Subsection \ref{ss:not}, we get
\begin{equation}\label{eq:Sdiag2}
j(e_\alpha)^+=-d_\alpha j(e_\alpha)^0  \qforq \alpha=5,6,7.
\end{equation}Hence, it is straightforward that if $(\alpha,\beta,\gamma)$ is an even permutation of $(5,6,7)$ then 
\begin{equation}\label{eq:br}
[j(e_{\alpha})^+,j(e_{\beta
})^+]=-2 d_\alpha d_\beta j(e_\gamma)^0 \qandq [\tau(e_\alpha),\tau(e_\beta)]=\tau(e_\gamma).
\end{equation}

We  use the expressions \eqref{eq:npend1}-\eqref{eq:npend2} to compute the components $(R^\lambda)^\alpha_\beta$  of the curvature in this basis (see \eqref{eq:Rlamab}). 

\begin{lm} \label{lm:Rab} The curvature $R^\lambda$ of the connection $\nabla^\lambda$ in \eqref{eq:nablap} satisfies:
\begin{enumerate}
\item\label{it:R_567} For $\alpha,\beta,\gamma\in\{5,6,7\}$, such that $(\alpha,\beta,\gamma)$ is an even permutation of $(5,6,7)$, we have:
\begin{multline}
\label{eq: R_567}
    (R^\lambda)^\alpha_\beta = \left(\frac{(\lambda+1)^2}{2}d_\alpha d_\beta-\frac{\lambda+1}{3}\mu\lambda(d_\alpha+d_\beta)+\frac{2}{9}(\lambda\mu)^2 \right)j(e_\gamma)^0-\frac{2}{3}\mu\lambda j(e_\gamma)\\
    -\frac{(\lambda-1)^2}{4}[j(e_\alpha)^-,j(e_\beta)^-]-\left(\frac{2}{3}\mu\lambda\right)^2\tau(e_\gamma).
\end{multline}

\item \label{it:R_14} For $\alpha,\beta\in \{1,2,3,4\}$ we have $(R^\lambda)^\alpha_\beta=\zeta^\alpha_\beta+\eta^\alpha_\beta$, where $\zeta^\alpha_\beta\in \Lambda^2\mc^*$ and $\eta^\alpha_\beta\in\Lambda^2\mq^*$ satisfy
\begin{multline}\label{eq:2-form_sigma}
\zeta^\alpha_\beta(e_j,e_k)
=
\left(
\frac{(\lambda-1)^2}4 d_jd_k-\frac{\lambda-1}6\lambda \mu( d_j+d_k)+(\frac13\mu\lambda)^2
\right)
[j(e_j)^0,j(e_k)^0]^\alpha_\beta+\frac{(\lambda+1)^2}4[j(e_j)^-,j(e_k)^-]^\alpha_\beta
\end{multline} for all $j,k\in\{5,6,7\}$,
and 
\begin{align}\label{eq:2-form_eta}
\eta^\alpha_\beta=&\sum_{l=5}^7 a_{l}
j(e_l)^{0}e_\alpha\wedge j(e_l)^0 e_\beta+ 
b_l
(j(e_l)^0e_\alpha\wedge j(e_l)^-e_\beta+j(e_l)^-e_\alpha\wedge j(e_l)^0e_\beta)\\ \nonumber
&-\left(c_l(j(e_l)^0)^\beta_\alpha+\frac{\lambda+1}2 (j(e_l)^-)^\beta_\alpha\right)j(e_l)
-\frac{(\lambda-1)^2}4 j(e_l)^{-}e_\alpha\wedge j(e_l)^- e_\beta
\end{align}
where 
\begin{equation}\label{eq:abc}
a_{l}:=-\left(\frac{\mu\lambda}{3}-\frac{\lambda+1}2d_l\right)^2 \quad
b_{l}:= \frac{\lambda-1}{2}\left(\frac{\mu\lambda}{3}-\frac{\lambda+1}2d_l\right) \quad
c_{l}:= \frac{\lambda-1}2 d_l  -\frac{\mu\lambda}3.
\end{equation}
\end{enumerate}
\end{lm}

Notice that in some of the above expressions, we mix 2-forms with skew-symmetric maps on behalf of the identifications $\Lambda^2\mq\simeq \so(\mq)$. Also, if $A\in\so(\mq)$ and $\alpha,\beta\in\{1, \ldots,4\}$ we denote $A^\alpha_\beta:=\lela Ae_\beta,e_\alpha\rira$.

\begin{proof}
Firstly, for any $x,y\in \mg$ we compute the endomorphism $R^\lambda_{x,y}\in \so(\mg)$. Thus, 
for $u,v\in\mq$  set $\epsilon_u, \epsilon_v\in\{\pm,0\}$, we know that ${\rm Im}\ad_u^{\epsilon_u}\subset \mz\subset \ker(\ad_v^{\epsilon_v})$, hence $[\ad_u^{\epsilon_u},\ad_v^{\epsilon_v}]=0=[(\ad_u^{\epsilon_u})^*,(\ad_v^{\epsilon_v})^*]$. These facts together with \eqref{eq:npend1} give
\begin{eqnarray}
R_{u,v}^\lambda&=&-\frac{(\lambda+1)^2}4\left([\ad_u^+,(\ad_v^+)^*]+[(\ad_u^+)^*,\ad_v^+]\right)
-\frac{(\lambda-1)^2}4\left([\ad_u^-,(\ad_v^-)^*]+[(\ad_u^-)^*,\ad_v^-]\right)\nonumber\\
&&+\frac{\lambda^2-1}{4}\left([\ad_u^-,(\ad_v^+)^*]+[(\ad_u^-)^*,\ad_v^+]\right)
+\frac{\lambda^2-1}{4}\left([\ad_u^+,(\ad_v^-)^*]+[(\ad_u^+)^*,\ad_v^-]\right)\nonumber\\
&&-\frac{(\lambda+1)}6\mu\lambda\left([\ad_u^+,(\ad_v^0)^*]+[(\ad_u^+)^*,\ad_v^0]+[\ad_u^0,(\ad_v^+)^*]+[(\ad_u^0)^*,\ad_v^+]\right)\label{eq:Ruv}\\
&&+\frac{(\lambda-1)}6\mu\lambda\left([\ad_u^-,(\ad_v^0)^*]+[(\ad_u^-)^*,\ad_v^0]+[\ad_u^0,(\ad_v^-)^*]+[(\ad_u^0)^*,\ad_v^-]\right)\nonumber\\
&&-\frac19(\lambda\mu)^2\left([\ad_u^0,(\ad_v^0)^*]+[(\ad_u^0)^*,\ad_v^0]\right)\nonumber\\
&&+\frac{(\lambda-1)}2j([u,v])^++\frac{(\lambda+1)}2j([u,v])^--\frac{1}3\mu\lambda \, j([u,v])^0+\frac23\mu\lambda \tau([u,v]),\nonumber
\end{eqnarray}
where $\tau$ is as in \eqref{eq:taudef}.
In addition, for any  $z,z'\in \mc$, and $u,v,w\in\mq$ we have
\begin{eqnarray}
\label{eq:epuepv}
[\ad_u^{\epsilon_u},(\ad_v^{\epsilon_v})^*]w&=&-(\ad_v^{\epsilon_v})^* \ad_u^{\epsilon_u}w=-j([u,w]^{\epsilon_u})^{\epsilon_v}v,\\
\lela [\ad_u^{\epsilon_u},(\ad_v^{\epsilon_v})^*]z,z'\rira&=&\lela \ad_u^{\epsilon_u}(\ad_v^{\epsilon_v})^*z,z'\rira=-\lela j(z)^{\epsilon_v} j(z')^{\epsilon_u} u,v\rira,
\nonumber\\
\lela [\ad_u^{\epsilon_u},(\ad_v^{\epsilon_v})^*]z,w\rira&=&0.\nonumber
\end{eqnarray}
Hence, the endomorphism $R_{u,v}^\lambda$ preserves $\mq$ and $\mc$.

 Now, let $z,z'\in\mc$, by \eqref{eq:jzconmut} and \eqref{eq:npend2} we get
\begin{eqnarray}
 R^\lambda_{z,z'}
  &=&\frac{(\lambda-1)^2}4[j(z)^+,j(z')^+]+\frac{(\lambda-1)}6\lambda\mu ([j(z)^+,j(z')^0]+[j(z)^0,j(z')^+])\label{eq:Rzz}\\
&& 
 +(\frac13\mu\lambda)^2[j(z)^0,j(z')^0]+(\frac23\lambda\mu)^2[\tau(z),\tau(z')]+\frac{(\lambda+1)^2}4[j(z)^-,j(z')^-].\nonumber
\end{eqnarray}

Finally, for $u\in\mq$,  and $z\in\mc$ set  $\epsilon_u,\epsilon_z\in\{\pm,0\}$. Thus,  for any $(w,z')\in \mq\oplus\mc$ one has the following:
\begin{eqnarray*}
\,[\ad_u^{\epsilon_u}-(\ad_u^{\epsilon_u})^*,j(z)^{\epsilon_z}](w,z') &=& ( j(z)^{\epsilon_z}(\ad_u^{\epsilon_u})^*z',
\ad_u^{\epsilon_u}j(z)^{\epsilon_z} w)\in \mq\oplus\mc\\
\,[\ad_u^{\epsilon_u}-(\ad_u^{\epsilon_u})^*,\tau(z)] (w,z') &=& -((\ad_u^{\epsilon_u})^*\tau(z)z' ,
\tau(z) \ad_u^{\epsilon_u}w)\in \mq\oplus\mc.
\end{eqnarray*}
Hence, from \eqref{eq:npend1} and \eqref{eq:npend2}, we get that $R_{u,z}\in \mq\otimes\mc$, which means 
\begin{equation}\label{eq:Ruzzz}
\lela R_{u,z}^\lambda z',z''\rira =\lela R_{u,z}^\lambda w,w'\rira  =0. \qforq z',z''\in \mc \qandq w,w'\in \mq.
\end{equation}

Next, we compute the component of the curvature  $(R^\lambda)^\beta_\alpha\in \Lambda^2\mg^*$ defined in \eqref{eq:Rlamab} for each case in the statement. \smallskip

\noindent  Case (1): For $5\leq \alpha <\beta\leq 7$, equation \eqref{eq:Ruv} together with  \eqref{eq:jzconmut} and \eqref{eq:epuepv} give
\begin{align}\nonumber
\lela R_{u,v}^\lambda e_\beta,e_\alpha\rira=&\lela \left(-\frac{(\lambda+1)^2}4 [j(e_\alpha)^+,j(e_\beta)^+]
-\frac{(\lambda-1)^2}4 [j(e_\alpha)^-,j(e_\beta)^-]-\frac19\lambda^2\mu^2 [j(\alpha)^0,j(\beta)^0]\right.\right.\\
&\left.\left.-\frac{(\lambda+1)}6\mu\lambda ([j(e_\alpha)^0,j(e_\beta)^+]+[j(e_\alpha)^+,j(e_\beta)^0])\right)u,v\rira +\frac23  \mu\lambda \lela \tau([u,v])e_\beta,\alpha\rira. \label{eq:Ruvzz}
\end{align} 
In addition, from \eqref{eq:Rzz} we obtain
\begin{equation}\label{eq:Rzzzz}
\lela R_{z,z'}^\lambda e_\beta,e_\alpha\rira =(\frac23 \mu\lambda)^2\lela [\tau(z),\tau(z')]e_\beta,e_\alpha\rira, \forall z,z'\in\mc.
 \end{equation}
Thus, using that $\tau(z)z'=-\tau(z')z$ and $\lela \tau([u,v])z,z'\rira=\lela j(\tau(z)z')u,v\rira$ for every $z,z'\in\mc$ and $u,v\in\mq$,  from \eqref{eq:Ruvzz} and \eqref{eq:Rzzzz},  we get 
\begin{align}\label{eq:r56}
(R^\lambda)^\alpha_\beta=&-\frac{(\lambda+1)^2}4 [j(e_\alpha)^+,j(e_\beta)^+]
-\frac{(\lambda-1)^2}4 [j(e_\alpha)^-,j(e_\beta)^-]-\frac19\lambda^2\mu^2 [j(e_\alpha)^0,j(e_\beta)^0]
\\ \nonumber
&-\frac{(\lambda+1)}6\mu\lambda ([j(e_\alpha)^0,j(e_\beta)^+]+[j(e_\alpha)^+,j(e_\beta)^0])
+\frac23\mu\lambda \left(j(\tau(e_\beta)e_\alpha)+\frac23\mu\lambda [\tau(e_\beta),\tau(e_\alpha)]\right);
\end{align}
here we are identifying $\Lambda^2\mc^*\simeq\so(\mc)$ and $\Lambda^2\mq^*\simeq \so(\mq)$. 
Thus, the formula \eqref{eq: R_567} follows by applying \eqref{eq:Sdiag2} and \eqref{eq:br} into \eqref{eq:r56}.

\smallskip
\noindent Case (2): Now, for  $1\leq \alpha<\beta\leq 4$, 
we have
$(R^\lambda)^\alpha_\beta\in \Lambda^2\mq^*\oplus \Lambda^2\mc^*$ by \eqref{eq:Ruzzz}. Write $(R^\lambda)^\alpha_\beta=\eta^\alpha_\beta+\zeta^\alpha_\beta$, where $\eta^\alpha_\beta\in \Lambda^2\mq^*$ and $\zeta^\alpha_\beta\in \Lambda^2\mc^*$. By \eqref{eq:Rzz} and \eqref{eq:Sdiag2}, for any $e_j,e_k\in \mc$ we get obtain
\begin{align}\nonumber
\zeta^\alpha_\beta(e_j,e_k)=& \lela R^\lambda_{e_j,e_k}e_\beta, e_\alpha\rira \\ \label{eq:sab}
=&\left(\frac{(\lambda-1)^2}4 d_jd_k+\frac{(\lambda-1)}6\lambda\mu(d_j+d_k)+(\frac13\mu\lambda)^2\right)\lela
[j(e_j)^0,j(e_k)^0]e_\beta,e_\alpha \rira\\ \nonumber
&+\frac{(\lambda+1)^2}4\lela[j(e_j)^-,j(e_k)^-]e_\beta,e_\alpha\rira.
\end{align}
Furthermore, for any $e_j,e_k\in\mq$ the equation  \eqref{eq:Ruv} becomes
\begin{align}\label{eq:tab}\begin{split}
\eta^\alpha_\beta(e_j,e_k)=& \lela R^\lambda_{e_j,e_k}e_\beta, e_\alpha\rira\\
=&-\frac{(\lambda+1)^2}4\left(-\lela[e_j,e_\beta]^+,[e_k,e_\alpha]^+\rira+\lela [e_k,e_\beta]^+,[e_j,e_\alpha]^+\rira\right)\\
&-\frac{(\lambda-1)^2}4\left(-\lela [e_j,e_\beta]^-,[e_k,e_\alpha]^-\rira
+\lela [e_k,e_\beta]^-,[e_j,e_\alpha]^-\rira\right)\\
&-\frac19(\lambda\mu)^2\left(-\lela[e_j,e_\beta]^0,[e_k,e_\alpha]^0\rira
+\lela [e_k,e_\beta]^0,[e_j,e_\alpha]^0\rira\right)\\
&+\frac{\lambda^2-1}{4}\left(-\lela [e_j,e_\beta]^-,[e_k,e_\alpha]^+\rira
+\lela [e_k,e_\beta]^+,[e_j,e_\alpha]^-\rira\right.\\
&\left.-\lela [e_j,e_\beta]^+,[e_k,e_\alpha]^-\rira
+\lela [e_k,e_\beta]^-,[e_j,e_\alpha]^+\rira\right)\\
&-\frac{1+\lambda}{6}\lambda\mu\left(-\lela [e_j,e_\beta]^+,[e_k,e_\alpha]^0\rira
+\lela [e_k,e_\beta]^0,[e_j,e_\alpha]^+\rira\right)\\
&\left.-\lela [e_j,e_\beta]^0,[e_k,e_\alpha]^+\rira
+\lela [e_k,e_\beta]^+,[e_j,e_\alpha]^0\rira\right)\\
&+\frac{(\lambda-1)}{6}\lambda\mu\left(-\lela [e_j,e_\beta]^-,[e_k,e_\alpha]^0\rira
+\lela [e_k,e_\beta]^0,[e_j,e_\alpha]^-\rira\right)\\
&\left.-\lela [e_j,e_\beta]^0,[e_k,e_\alpha]^-\rira
+\lela [e_k,e_\beta]^-,[e_j,e_\alpha]^0\rira\right)\\
&+\frac{(1-\lambda)}2\lela j([e_j,e_k])^+e_\beta,e_\alpha\rira +\frac{(\lambda+1)}2\lela j([e_j,e_k])^-e_\beta,e_\alpha\rira\\
&-\frac13\mu\lambda\lela j([e_j,e_k])^0 e_\beta,e_\alpha\rira .
\end{split}
\end{align}
Now, notice that for any $\epsilon_1,\epsilon_2\in\{\pm,0\}$, 
\begin{align}\label{eq:mod}
\lela[e_k,e_\beta]^{\epsilon_1},[e_j,e_\alpha]^{\epsilon_2}\rira=&\sum_{l=5}^7 (\lela[e_k,e_\beta]^{\epsilon_1},e_l\rira\lela e_l,[e_j,e_\alpha]^{\epsilon_2}\rira\\ \nonumber
=&\sum_{l=5}^7(\lela j(e_l)^{\epsilon_1}e_\beta,e_k\rira\lela j(e_l)^{\epsilon_2}e_\alpha,e_j\rira.
\end{align}
In particular, if $\epsilon_1=\epsilon_2=\epsilon$,
\begin{equation}\label{eq:epig}
\lela[e_k,e_\beta]^{\epsilon},[e_j,e_\alpha]^{\epsilon}\rira-\lela[e_j,e_\beta]^{\epsilon},[e_k,e_\alpha]^{\epsilon}\rira
= \sum_{l=5}^7j(e_l)^{\epsilon}e_\alpha\wedge j(e_l)^\epsilon e_\beta(e_j,e_k).
\end{equation}
In addition, if $\epsilon_1=-$ and $\epsilon_2=+$ in \eqref{eq:mod}, we get
\begin{multline}
\lela [e_j,e_\beta]^-,[e_k,e_\alpha]^+\rira-\lela [e_k,e_\beta]^+,[e_j,e_\alpha]^-\rira
+\lela [e_j,e_\beta]^+,[e_k,e_\alpha]^-\rira
-\lela [e_k,e_\beta]^-,[e_j,e_\alpha]^+\rira\\
=\sum_{l=5}^7-\lela j(e_l)^{+}e_\beta,e_k\rira\lela j(e_l)^{-}e_\alpha,e_j\rira+\lela j(e_l)^{+}e_\beta,e_j\rira\lela j(e_l)^{-}e_\alpha,e_k\rira\\
-\lela j(e_l)^{-}e_\beta,e_k\rira\lela j(e_l)^{+}e_\alpha,e_j\rira+\lela j(e_l)^{-}e_\beta,e_j\rira\lela j(e_l)^{+}e_\alpha,e_k\rira\\
= e_k\lrcorner e_ j\lrcorner\left(\sum_{l=5}^7j(e_l)^+e_\beta\wedge j(e_l)^-e_\alpha+j(e_l)^-e_\beta\wedge j(e_l)^+e_\alpha\right).
\end{multline}Similar expressions are obtained for $\epsilon_1\neq\epsilon_2\in\{\pm,0\}$. Using these, \eqref{eq:epig} and \eqref{eq:Sdiag2} in \eqref{eq:tab} we get
\begin{multline}\label{eq:tauab}
\eta^\alpha_\beta=\sum_{l=5}^7
\left(-\frac{(\lambda+1)^2}4d_l^2+\frac{\lambda+1}{3}d_l\mu\lambda-\frac19(\mu\lambda)^2\right)  j(e_l)^{0}e_\alpha\wedge j(e_l)^0 e_\beta-\frac{(\lambda-1)^2}4 j(e_l)^{-}e_\alpha\wedge j(e_l)^- e_\beta\\
-\sum_{l=5}^7\left(\frac{\lambda^2-1}{2}d_l
-\frac{\lambda+1}{3}\mu\lambda d_l\right)(j(e_l)^0e_\alpha\wedge j(e_l)^-e_\beta+j(e_l)^-e_\alpha\wedge j(e_l)^0e_\beta)
\\
+\sum_{1\leq j<k\leq 4}\left\{\frac{(1-\lambda)}2\lela j([e_j,e_k])^+e_\beta,e_\alpha\rira+\frac{(\lambda+1)}2\lela j([e_j,e_k])^-e_\beta,e_\alpha\rira\right.\\
\left.-\frac13\mu\lambda\lela j([e_j,e_k])^0 e_\beta,e_\alpha\rira\right\}e^j\wedge e^k.
\end{multline}
Finally, notice that for   $\epsilon\in\{0,\pm\}$, using \eqref{eq:brbeta} and \eqref{eq:j0} we get  
\begin{equation*}
\lela j([e_j,e_k])^\epsilon e_\beta,e_\alpha\rira= \sum_{l=5}^7 \lela [e_j,e_k],[e_\beta,e_\alpha]^\epsilon\rira e^j\wedge e^k= \sum_{l=5}^7 \lela j(e_l)^\epsilon e_\beta,e_\alpha\rira \lela j(e_l)e_j,e_k\rira.
\end{equation*}
Therefore, \eqref{eq:tauab} becomes
\begin{multline}\label{eq:tauab2}
\eta^\alpha_\beta=\sum_{l=5}^7
\left(-\frac{(\lambda+1)^2}4d_l^2+\frac{\lambda+1}{3}d_l\mu\lambda-\frac19(\mu\lambda)^2\right)  j(e_l)^{0}e_\alpha\wedge j(e_l)^0 e_\beta-\frac{(\lambda-1)^2}4 j(e_l)^{-}e_\alpha\wedge j(e_l)^- e_\beta\\
-\sum_{l=5}^7\left(\frac{\lambda^2-1}{2}d_l
-\frac{\lambda+1}{3}\mu\lambda d_l\right)(j(e_l)^0e_\alpha\wedge j(e_l)^-e_\beta+j(e_l)^-e_\alpha\wedge j(e_l)^0e_\beta)
\\
-\sum_{l=5}^7 \left\{\left(\frac{(\lambda-1)}2d_l-\frac13\mu\lambda \right)\lela j(e_l)^0 e_\alpha,e_\beta\rira +\frac{(\lambda+1)}2\lela j(e_l)^- e_\alpha,e_\beta\rira \right\}j(e_l).
\end{multline}

\end{proof}

\begin{lm}\label{lm:etaab}
    Let $\eta^\alpha_\beta\in \Lambda^2\mq^*$ be the $2$-form given in \eqref{eq:2-form_eta}. Thus, its self-dual part is
\begin{equation*}
(\eta^\alpha_\beta)^+ =
\sum_{l=5}^7  \left\{\left({\frac12}
(a_l-\sum_{\begin{smallmatrix}l=5,6,7,\\s\neq l\end{smallmatrix}}a_s)-\frac{(\lambda-1)^2}8m^2{+}c_{l}d_{l}
\right)(j(e_l)^0)^\beta_\alpha+\left(b_l+{d_l}\frac{(\lambda+1)}{{2}}\right) ( j(e_l)^-)^\beta_\alpha \right\}\sigma_{l-4}^+,
 \end{equation*}
where $m_l\in \R$ satisfies $(j(e_l)^-)^2=-m_l\Id$ for $l=5,6,7$ and $m^2:=m_1^2+m_2^2+m_3^2$.
\end{lm}
\begin{proof}
Consider the following 2-form in $\Lambda^2\mq^*$:
\[
\sum_{l=5}^7 p_l\; j(e_l)^{\epsilon_1}e_\alpha\wedge j(e_l)^{\epsilon_2} e_\beta,
\]where $\epsilon_1,\epsilon_2\in\{0,\pm\}$ and $p_l\in\R$.
For any fixed $k=1,2,3$, let $\sigma^+_k$ be the self-dual form in \eqref{eq:sigmabasis}. Canonical computations give
\begin{equation*}
\lela \sum_{l=5}^7p_l j(e_l)^{\epsilon_1} e_\alpha\wedge j(e_l)^{\epsilon_2} e_\beta,\sigma_k^+\rira=\sum_{l=5}^7 p_l\sum_{s,t=1}^4\lela   j(e_l)^{\epsilon_1}e_\alpha,e_s\rira\lela j(e_l)^{\epsilon_2} e_\beta,e_t\rira\lela e^{s}\wedge e^{t},\sigma_k^+\rira.
\end{equation*}Notice that $\lela e^{s}\wedge e^{t},\sigma_k^+\rira= e_t\lrcorner e_s\lrcorner \sigma_k^+=\sigma_k^+(e_s,e_t)=\lela j(e_{k+4})^0e_s,e_t\rira$, for any $s,t=1, \ldots, 4$.
Hence
\begin{align}\label{eq:jjeps}
\lela \sum_{l=5}^7 p_lj(e_l)^{\epsilon_1}e_\alpha\wedge j(e_l)^{\epsilon_2} e_\beta,\sigma_k^+\rira=&\sum_{l=5}^7 p_l \sum_{s=1}^4\lela   j(e_l)^{\epsilon_1}e_\alpha,e_s\rira\lela j(e_l)^{\epsilon_2} e_\beta,j(e_{k+4})^0e_s\rira\\ \nonumber
=&\sum_{l=5}^7p_l \lela j(e_l)^{\epsilon_2} e_\beta,j(e_{k+4})^0j(e_l)^{\epsilon_1}e_\alpha\rira.
\end{align}
First, take $\epsilon_i=-$ for $i=1,2$ and $p_l=-\frac{(\lambda-1)^2}{4}$ in this expression, for $l=5,6,7$. Since $j(e_{k+4})^0$ commutes with $j(e_l)^-$ and $(j(e_l)^-)^2=-m_l^2\Id$ for some constants $m_l$, for all $l=5,6,7$, we obtain
\begin{align*}
-\frac{(\lambda-1)^2}{4}\lela \sum_{l=5}^7  j(e_l)^{-}e_\alpha\wedge j(e_l)^- e_\beta,\sigma_k^+\rira
=&-\frac{(\lambda-1)^2}{4}\sum_{l=5}^7 \lela (j(e_l)^-)^2 e_\beta,j(e_{k+4})^0e_\alpha\rira\\
=&-\frac{(\lambda-1)^2}{4}(\sum_{l=5}^7 m_l^2)\sigma_k^+(e_\alpha,e_\beta)=-\frac{(\lambda-1)^2}{4}m^2\,\sigma_k^+(e_\alpha,e_\beta),
\end{align*}where we have set 
\begin{equation}\label{eq:mdef}
m^2:=\sum_{l=5}^7 m_l^2.
\end{equation}
Therefore, we get
\begin{equation}\label{eq:jj-SD}
-\frac{(\lambda-1)^2}{4}(\sum_{l=5}^7 j(e_l)^{-}e_\alpha\wedge j(e_l)^- e_\beta)^+=-\frac{(\lambda-1)^2}{8} m^2\sum_{k=1}^3\sigma_k^+(e_\alpha,e_\beta)\sigma_k^+.
\end{equation}
Now, consider $\epsilon_i=0$ for $i=1,2$ in \eqref{eq:jjeps}, we recall that $j(e_k)^0j(e_l)^0=-j(e_l)^0j(e_k)^0$ if $k\neq l\in \{5,6,7\}$, then
\begin{equation*}
\lela \sum_{l=5}^7 p_l j(e_l)^{0}e_\alpha\wedge j(e_l)^0 e_\beta,\sigma_k^+\rira
=p_{k+4}\lela  e_\beta,j(e_{k+4})^0e_\alpha\rira-\sum_{\begin{smallmatrix}l=5,6,7\\ l\neq k+4\end{smallmatrix}} p_l \lela  e_\beta,j(e_{k+4})^0e_\alpha\rira.
\end{equation*} So the self-dual part of $\sum_{l=5}^7 a_lj(e_l)^{0}e_\alpha\wedge j(e_l)^0 e_\beta$, for $a_l$ defined in \eqref{eq:abc} is 
\begin{multline}\label{eq:jj0SD}
(\sum_{l=5}^7 a_lj(e_l)^{0}e_\alpha\wedge j(e_l)^0 e_\beta)^+=\\\frac12 ((a_5-a_6-a_7)\sigma_1^+(e_\alpha,e_\beta)\sigma_1^++(a_6-a_5-a_7)\sigma_2^+(e_\alpha,e_\beta)\sigma_2^++(a_7-a_5-a_6)\sigma_3^+(e_\alpha,e_\beta)\sigma_3^+).
\end{multline}
Finally, we put $\epsilon_1=0$, $\epsilon_2=-$ in \eqref{eq:jjeps}. Using the fact that $j(e_l)^-j(e_{k+4})^0=-j(e_{k+4})^0j(e_l)^-$, for any $l=5,6,7$, $k=1,2,3$ (see \eqref{eq:j0} and \eqref{eq:jzconmut}), we get
\begin{multline*}
\lela
\sum_{l=5}^7 p_l(j(e_l)^0e_\alpha\wedge j(e_l)^-e_\beta+j(e_l)^-e_\alpha\wedge j(e_l)^0e_\beta),\sigma_k^+\rira\\
=\sum_{l=5}^7p_l(\lela j(e_l)^{-} e_\beta,j(e_{k+4})^0j(e_l)^{0}e_\alpha\rira+\lela j(e_l)^{0} e_\beta,j(e_{k+4})^0j(e_l)^{-}e_\alpha\rira)\\
=\sum_{l=5}^7p_l\lela j(e_l)^{-} e_\beta,(j(e_{k+4})^0j(e_l)^{0}+j(e_l)^{0}j(e_{k+4})^0)e_\alpha\rira=2p_{k+4}\lela j(e_{k+4})^{-} e_\alpha,e_\beta\rira.
\end{multline*}
Hence, taking $p_l=c_l$ as defined in \eqref{eq:abc}, we obtain
\begin{equation}\label{eq:jj-0SD}
(\sum_{l=5}^7 c_l(j(e_l)^0e_\alpha\wedge j(e_l)^-e_\beta+j(e_l)^-e_\alpha\wedge j(e_l)^0e_\beta))^+=\sum_{l=5}^7 c_l\lela j(e_l)^{-} e_\alpha,e_\beta\rira\sigma_{l-4}^+.
\end{equation}
Therefore, the result follows by replacing \eqref{eq:jj-SD}, \eqref{eq:jj0SD}, \eqref{eq:jj-0SD} and $j(e_l)^+=-d_lj(e_l)^0$ into \eqref{eq:2-form_eta}.
\end{proof}

\subsection{The Instanton Condition for \texorpdfstring{$\nabla^\lambda$}{D}.}\label{app:diagonal}

We continue with the notation and the assumption of the previous subsections.

By definition, the connection $\nabla^\lambda$ is a $\rG_2$-instanton if it satisfies \eqref{eq: instaton_equation_intro}.
In terms of the components $(R^\lambda)^\beta_\alpha$ of the curvature  in the basis $\{e_1, \ldots, e_7\}$ (see \eqref{eq:Rlamab}), we have that $\nabla^\lambda$ is an instanton if and only if 
$(R^\lambda)^\beta_\alpha\wedge \psi=0$ for every $\alpha,\beta=1, \ldots,7$.
We shall split this condition in three as follows:
\begin{eqnarray}
(R^\lambda)^\beta_\alpha\wedge \psi&=&0, \quad \mbox{ for }\alpha,\beta=5, \ldots,7.\label{eq:57}\\
(R^\lambda)^\beta_\alpha\wedge \psi&=&0, \quad \mbox{ for }\alpha,\beta=1, \ldots,4.\label{eq:14}\\
(R^\lambda)^\beta_\alpha\wedge \psi&=&0, \quad \mbox{ for }\alpha=1, \ldots,4, \beta=5,\ldots,7.\label{eq:1457}
\end{eqnarray}

In the sequel give necessary and sufficient conditions on the matrix $S={\rm diag}(d_5,d_6,d_7)$ and $\lambda$ for the above conditions to be satisfied. As in the previous subsection, $\mu$ denotes the trace of $S$, namely, $\mu=d_5+d_6+d_7$.

\begin{lm}\label{lm:insteigenv} Equation \eqref{eq:57} holds if and only if one of the following holds:
\begin{enumerate}
\item $S=0$ and $\lambda$ is arbitrary,
\item $S\neq 0$ with $\mu=0$ and $\lambda=-1$,
\item $S=\frac{\mu}3\Id$ with $\mu\neq 0$, and $\lambda=\{-\frac13,1\}$.
\item $S\neq 0$ with $\mu\neq 0$, and $S$ has exactly two eigenvalues $d_i=d_j$, $d_k=\frac12(\sqrt[3]{4}-2)d_j\neq 0$ and $\lambda=\sqrt[3]{4}-1$.
\end{enumerate}
\end{lm}

\begin{proof}The curvature components $(R^\lambda)^\alpha_\beta$ with $\alpha,\beta=5,6,7$ were computed in \eqref{eq: R_567} and $\psi$ has the form given in \eqref{eq:stphi}. 
Hence, if $(\alpha,\beta,\gamma)$ is an even permutation of $(5,6,7)$, we have
\begin{eqnarray*}
(R^\lambda)^\alpha_\beta\wedge\psi
&=&
\left((\lambda+1)^2 d_\alpha d_\beta-\frac{2}{3}(\lambda+1)\mu\lambda(d_\alpha+d_\beta)+\frac{4}{9}(\lambda\mu)^2 +\frac{4}{3}\mu\lambda d_\gamma -\left(\frac{2}{3}\mu\lambda\right)^2\right)e^{\alpha\beta}
\wedge e^{1234}.
\end{eqnarray*}
Therefore,  \eqref{eq:57} is equivalent to the system
\begin{equation}\label{eq:CSI0}
(\lambda+1)^2 C+2(\frac{\lambda}3+1)\lambda\mu S-\frac23(\lambda+1)\lambda\mu^2 \Id=0,
\end{equation} where $C=\diag(d_6d_7,d_7d_5,d_5d_6)$.

Notice that $S=0$ is a solutions of this system for any $\lambda$, so we assume $S\neq 0$ from now on. In addition, one can easily check that if $S\neq 0$ and $\lambda$ gives a solution of this system, then $\lambda=-1$ if and only if  $\mu=0$. In addition, $\lambda=0$ in \eqref{eq:CSI0} implies $S=0$. Hence, we need to find all possible $\lambda\neq 0$ that solve \eqref{eq:CSI0} for $S\neq 0$ with $\mu\neq 0$.

Multiplying \eqref{eq:CSI0} by $S$ and using that $CS=\det(S)\Id$, we get
\begin{equation*}
(\lambda+1)^2 \det(S) \Id +2(\frac{\lambda}3+1)\lambda\mu S^2-\frac23(\lambda+1)\lambda\mu^2 S=0.
\end{equation*}
By looking at the diagonal components of this matrix equation and taking the difference between each two of them, this system becomes equivalent to
\begin{equation}\label{eq:SSI}
(d_i-d_j)\mu((\lambda+1)d_k-2(d_i+d_j))=0, \qquad \mbox{ for all }i,j,k\in\{5,6,7\},\;i\neq j\neq k\neq i, 
\end{equation}

Assume that $S\neq 0$, with $\mu\neq 0$, and $\lambda$ solve \eqref{eq:CSI0}, then these satisfy  \eqref{eq:SSI} as well. This implies that either $d_5=d_6=d_7$ or there is some  $i,j\in\{5,6,7\}$ for which $d_i\neq d_j$.
In the former case we get $S=\frac{\mu}{3}\Id$ and thus $C=\frac{\mu^2}9\Id$ so \eqref{eq:CSI0} becomes
\[
\frac{\mu^2}9(3\lambda+1)(\lambda-1)=0.
\]Since $\mu\neq 0$, $\lambda$ is either $-\frac13$ or $1$.

Suppose now that $d_i\neq d_j$ for some $i,j$, then by \eqref{eq:SSI}, for  $k\neq i,j$, one gets
\begin{equation*}
d_k(\lambda+1)=2(d_i+d_j).
\end{equation*}
Since $\lambda\neq-1$ because $\mu\neq 0$, we obtain $d_k=2(d_i+d_j)/(\lambda+1)$. Replacing $d_k$ by this expression in \eqref{eq:SSI}, for $i,j\neq k$ we get
\[
0=(d_i-d_k)(\lambda+1)d_j-2(d_i+d_k)=\frac{(\lambda+3)}{(\lambda+1)}(\lambda-1)d_i-2d_j)((\lambda-1)d_j-2d_i).
\] One can easily verify that $\lambda=-3$ and $S\neq 0$ is not a solution of \eqref{eq:CSI0}. So we have either $d_i=(\lambda-1)d_j/2$ or $d_j=(\lambda-1)d_i/2$. If $d_i=(\lambda-1)d_j/2$ then $d_k=d_j\neq d_i$, and if $d_j=(\lambda-1)d_i/2$ then $d_k=d_i\neq d_j$. In particular, $S$ has exactly 2 eigenvalues. 

Without loss of generality, we may assume $d_5=d_6$ and $d_5\neq d_7=(\lambda-1)d_5/2$. Notice that $d_5\neq 0$ since otherwise we would have $d_5=d_6=d_7$. 
Using the equalities $d_5=d_6$ and $d_5\neq d_7=(\lambda-1)d_5/2$ in \eqref{eq:CSI0}, the difference between the last two equations in the diagonal imply
\[
d_5^2(\lambda-3)(\lambda^3+3\lambda^2+3\lambda-3)=0.
\] However, $d_5\neq0$ and $\lambda\neq 3$ since, in both cases, by \eqref{eq:CSI0}, we would $d_5=0$ an thus $d_5=d_7$ contradicting our assumption. Therefore, $\lambda$ satisfies $\lambda^3+3\lambda^2+3\lambda-3=0$.

The only real root of $\lambda^3+3\lambda^2+3\lambda-3$ is $\lambda=\sqrt[3]{4}-1$. It is possible to check that $\lambda=\sqrt[3]{4}-1$, $d_6=d_5$ and $d_7=\frac12(\lambda-1)d_5=\frac12(\sqrt[3]{4}-2)d_5$ satisfy \eqref{eq:CSI0}.

\end{proof}

The following result shows, in particular, that if the first two equations for the instanton condition hold, then $S$ is a non-zero multiple of the identity.

\begin{lm}\label{lm:reduc2}
Equations \eqref{eq:57} and \eqref{eq:14} hold simultaneously if and only if the following conditions hold:
\begin{enumerate}
\item \label{it:lam} either $\lambda=1$, or both $\lambda=-\frac13$ and $j(z)^-=0$ for all $z\in\mc$,
\item \label{it:Sid} $S=\frac{\mu}3\Id$ with $\mu\neq 0$, 
\item \label{it:jc} the map $j:\mc\to\so(\mq)$ satisfies
\begin{equation}\label{eq:jcond}
[j(z),j(z')]=\frac23 \mu j(\tau(z)z'), \qquad \forall z,z'\in\mc,
\end{equation}where $\tau\in\so(\mc)$ is defined by in \eqref{eq:taudef}.
\end{enumerate}
\end{lm}

\begin{proof}
First, we shall write a system of equations which is equivalent to \eqref{eq:14}. By \ref{it:R_14} in Lemma \ref{lm:Rab}, for any $\alpha,\beta\in\{1, \ldots, 4\}$,  $(R^\lambda)^\alpha_\beta=\zeta^\alpha_\beta+\eta_\beta^\alpha$ where $\zeta^\alpha_\beta\in\Lambda^2\mc^*$ and $\eta_\beta^\alpha\in\Lambda^2\mq^*$.  We compute separately $\eta^\alpha_\beta\wedge \psi$ and $\zeta^\alpha_\beta\wedge \psi$.

Since $\psi$ has the form \eqref{eq:stphi} and $\eta^\alpha_\beta\in\Lambda^2\mq^*$ we have $\eta^\alpha_\beta\wedge e^{1234}=0$  and thus
\begin{eqnarray}
    \label{eq:etwsp}
\eta^\alpha_\beta\wedge \psi &=&\eta^\alpha_\beta\wedge(\sigma_1^+\wedge e^{67}+\sigma_2^+\wedge e^{75}+\sigma_3^+\wedge e^{56})\\
&=&(\eta^\alpha_\beta)^+\wedge(\sigma_1^+\wedge e^{67}+\sigma_2^+\wedge e^{75}+\sigma_3^+\wedge e^{56}).\non
\end{eqnarray} Besides, $(\eta_\beta^\alpha)^+$ is given in Lemma \ref{lm:etaab}, from which we obtain 
\begin{multline*}
(\eta^\alpha_\beta)^+ =
\sum_{l=5}^7  \left\{\left({\frac12}
(a_l-\sum_{\begin{smallmatrix}l=5,6,7,\\s\neq l\end{smallmatrix}}a_s)-\frac{(\lambda-1)^2}8m^2{+}c_{l}d_{l}
\right)(j(e_l)^0)^\beta_\alpha+\left(b_l+{d_l}\frac{(\lambda+1)}{{2}}\right) ( j(e_l)^-)^\beta_\alpha \right\}\sigma_{l-4}^+
 \end{multline*}and therefore, 
 the wedge product in \eqref{eq:etwsp} equals to
  \begin{multline}
     \eta^\alpha_\beta \wedge \psi =
  \sum_{l=5}^7  \left\{\left(
(a_l-\sum_{\begin{smallmatrix}l=5,6,7,\\s\neq l\end{smallmatrix}}a_s)-\frac{(\lambda-1)^2}4m^2{+}c_{l}d_{l}
\right)(j(e_l)^0)^\beta_\alpha\right.\\
\hspace{2cm}\left.\phantom{\sum_{\begin{smallmatrix}l=5,6,7,\\s\neq l\end{smallmatrix}}}
+\left(2b_l+{d_l}{(\lambda+1)}\right) ( j(e_l)^-)^\beta_\alpha\right\}e^{1234rt} \label{eq:ewsp},
\end{multline}
where $(l,r,t)$ is an even permutation of $(5,6,7)$.

 Now we compute $\zeta^\alpha_\beta\wedge \psi$, where $\zeta^\alpha_\beta$ is given in \eqref{it:R_14} of Lemma \ref{lm:Rab}. Using  this together with the expression of $\psi$ as in \eqref{eq:stphi}, we get
\begin{multline}\label{eq:swsp}
    \zeta^\alpha_\beta\wedge \psi=
\sum_{5\leq j<k\leq 7}\left\{
\left(-
\frac{(\lambda-1)^2}2 d_jd_k+\frac{\lambda-1}3\lambda \mu( d_j+d_k)-2(\frac13\mu\lambda)^2
\right)
(j(e_r)^0)^\alpha_\beta\right.\\
\left.
+\frac{(\lambda+1)^2}4[j(e_j)^-,j(e_k)^-]^\alpha_\beta\right\}e^{1234jk},
\end{multline}where $(j,k,r)$ is an even permutation of $(5,6,7)$.

Clearly, \eqref{eq:14}  is equivalent to 
 the equality $ - (\eta^\alpha_\beta)^+ \wedge \psi  = \zeta^\alpha_\beta\wedge \psi$ for every $\alpha,\beta\in\{1,\ldots, 4\}$. By \eqref{eq:ewsp} and \eqref{eq:swsp}, this equality holds if and only if for every even permutation  $(l,j,k)$  of $(5,6,7)$ one has
 \begin{multline}
   \left(
(a_l-\sum_{\begin{smallmatrix}l=5,6,7,\\s\neq l\end{smallmatrix}}a_s)-\frac{(\lambda-1)^2}4m^2{+}c_{l}d_{l}
\right)j(e_l)^0+\left(2b_l+{d_l}{(\lambda+1)}\right) j(e_l)^-=\\
=
\left(-
\frac{(\lambda-1)^2}2 d_jd_k+\frac{\lambda-1}3\lambda \mu( d_j+d_k)-2(\frac13\mu\lambda)^2
\right)
j(e_{l})^0+\frac{(\lambda+1)^2}4[j(e_j)^-,j(e_k)^-].
  \end{multline}
  By looking at the  at the self-dual and anti-self-dual components of this system we get that \eqref{eq:14} holds if and only if for every even permutation $(l,j,k)$ of $(5,6,7)$ the following systems holds
 \begin{eqnarray} 
 \label{eq:sys1}
(a_l-\sum_{\begin{smallmatrix}l=5,6,7,\\s\neq l\end{smallmatrix}}a_s)-\frac{(\lambda-1)^2}4m^2{+}2c_{l}d_{l}
&=&-
\frac{(\lambda-1)^2}2 d_jd_k+\frac{\lambda-1}3\lambda \mu( d_j+d_k)-\frac29(\mu\lambda)^2\\
 \label{eq:sys2}\left(2b_l+{d_l}{(\lambda+1)}\right) j(e_l)^-&=&\frac{(\lambda+1)^2}4[j(e_j)^-,j(e_k)^-].
 \end{eqnarray}

Consequently,  \eqref{eq:57} and \eqref{eq:14} hold simultaneously if and only if $S$ and $\lambda$ are as listed in Lemma  \ref{lm:insteigenv} and also satisfy \eqref{eq:sys1} and \eqref{eq:sys2} above. 
The last part of the proof consist in assuming each possibility for $S$ and $\lambda$ listed in Lemma  \ref{lm:insteigenv} and analyzing the solutions of \eqref{eq:sys1} and \eqref{eq:sys2}.

Assume (1) in Lemma \ref{lm:insteigenv} holds, that is,  $S=0$. The values in \eqref{eq:abc} verify $a_l=b_l=c_l=0=\mu$ and  $j(z)=j(z)^-$ for all $z\in\mc$ by \eqref{eq:Sdiag2}. In particular, the image of $j:\mc\to\so(\mq)$ is contained in $\so(\mq)^-\simeq\so(3)$. With these coefficients, \eqref{eq:sys1} reduces to the sole equation
\[
0=(\lambda-1)^2m^2.
\] Hence either $\lambda=1$ or $m=0$. By \eqref{eq:mdef}, the latter is equivalent to $j^-(z)=0$ which, by our assumption, is equivalent to $\mg$ abelian, leading to a contradiction. Also, if $\lambda=1$,
\eqref{eq:sys2}  is equivalent to 
the image of the map $j:\mc\to\so(\mq)^-$ being abelian. Since $\so(\mq)^-\simeq\so(3)$, the latter is possible only if $\dim {\rm Im} j\leq 1$. However, $\dim \mg'\geq 2$,    $\mg'\subset \mc$ and the fact that $\ker j\cap \mg'=0$ implies  $\dim {\rm Im} j\geq 2$. 
Therefore, \eqref{eq:sys2} does not hold for $S=0$ (and any $\lambda\in\R$).

Now suppose that $S\neq 0$ with $\mu\neq 0$ and $\lambda=-1$, that is, (2) in Lemma \ref{lm:insteigenv} holds. In this case, the constants in \eqref{eq:abc} are $a_l=0=b_l$ and $c_l=-d_l$. Hence, replacing these values in 
\eqref{eq:sys1} one obtains
\[
-m^2-2d_l=-2d_jd_k
\]for every even permutation $(l,j,k)$ of $(5,6,7)$. Using that $d_7=-(d_5+d_6)$ in the equations obtained for $(l,j,k)=(5,6,7)$ and $(6,7,5)$, and adding those, we get
\[
-(m^2+d_5^2+d_6^2)=(d_5+d_6)^2
\]so we must have $m=0$ and $d_5=d_6=d_7=0$. Hence $S\neq 0$ with $\mu\neq0$ and $\lambda=-1$ does not satisfy the system above. 

Now, let $S$ be a symmetric matrix satisfying (4) in Lemma \ref{lm:insteigenv}, for simplicity, assume $d_5=d_6$, $d_7=\frac12(\sqrt[3]{4}-2)d_5\neq 0$ and $\lambda=\sqrt[3]{4}-1$; the other cases are analogous.

In this case, the system \eqref{eq:sys1} reduces to only two equation. Taking their difference, we get $d_5^2=0$, which implies $S=0$ contradicting the fact that $S$ has two different eigenvalues.



Finally, suppose $S=\frac{\mu}3\Id$ with $\mu\neq 0$, i.e. (3) in Lemma \ref{lm:insteigenv} holds. Notice first that from \eqref{eq:Sdiag2} we have $j(e_{i+4})^+=-\frac{\mu}3j(e_{i+4})^0$ for $i=1,2,3$, and thus, by \eqref{eq:br}, one has
\begin{equation}\label{eq:jj+}
[j(z)^+,j(z')^+]=\frac23\mu\, j(\tau(z)z')^+,\qquad \forall z,z'\in\mc,
\end{equation}
 where $\tau\in\so(\mc)$ is as in \eqref{eq:taudef}. Since for any $z,z'\in\mc$,  $j(z)=j(z)^++j(z)^-$ and $[j(z)^+,j(z')^-]=0$, it is clear from \eqref{eq:jj+} that \eqref{eq:jcond} holds if and only if 
\begin{equation}\label{eq:jcondm}
[j(z)^-,j(z')^-]=\frac23\mu \,j(\tau(z)z')^-, \qquad \forall z,z'\in\mc.
\end{equation}

Now, assume further that $\lambda=-\frac13$,
then the coefficients in \eqref{eq:abc} verify $a_l=-\frac4{81}\mu^2$, $b_l=\frac4{27}\mu$ and $c_l=-\frac19\mu$, for all $l=1,2,3$. Hence  \eqref{eq:sys1} reduces to $m=0$, which is equivalent to $j(z)^-=0$ for all $z\in\mc$ due to \eqref{eq:mdef}. Notice that the later automatically implies \eqref{eq:sys2}.
Therefore, when $S=\frac{\mu}3\Id$ with $\mu\neq 0$ and $\lambda=-\frac13$, \eqref{eq:sys1}-\eqref{eq:sys2} hold if and only if $j(z)^-=0$ for all $z\in\mc$. If this is the case, \eqref{eq:jcondm} holds trivially and thus \eqref{eq:jcond} is valid.

It remains to treat the case $S=\frac{\mu}3\Id$ with $\mu\neq 0$ and $\lambda=1$. In this situation, $a_l=b_l=0$ and $c_l=-\frac13\mu$, for all $l=1,2,3$, and with these values, it is easy to check that \eqref{eq:sys1} is automatically satisfied. 
Moreover,  \eqref{eq:sys2} becomes
\begin{eqnarray}
{}[j(e_j)^-,j(e_k)^-]=\frac23 \mu\, j(e_l)^-, 
\end{eqnarray}
for all $(l,j,k)$ even permutation of $(5,6,7)$, which is equivalent to \eqref{eq:jcondm} and thus to \eqref{eq:jcond}.

This finishes the proof of the equivalence between (1)-(2)-(3) and \eqref{eq:57}-\eqref{eq:14}.
\end{proof}

A straightforward consequence of this result is that $\tr S=\mu\neq0$ when the connection $\nabla^\lambda$ is an instanton. This implies that the  torsion form $\tau_0$ of $\varphi$ is non-zero, because of \eqref{eq:t0}. Hence we get:
\begin{cor}\label{cor:notpcc}
If $\nabla^\lambda$ is an instanton, then $\varphi$ is not purely coclosed.
\end{cor}

From Lemma \ref{lm:reduc2}, we can also deduce some structural properties of the Lie algebras $\mg$ for which $\nabla^\lambda$ is an instanton.  

\begin{cor}\label{cor:g=c}If $\nabla^\lambda$ is an instanton, then $\dim\mg'=3$  (and thus $\mc=\mg'$ and $\mq=\mr$).
\end{cor}
\begin{proof}This is a consequence of  \eqref{it:jc} in Lemma \ref{lm:reduc2}. Indeed, suppose that $j:\mc\to\so(\mq)$ satisfies \eqref{eq:jcond} for some $\mu\neq 0$ and $\mg'\subsetneq \mc$. Then, from \S \ref{sec:prelim2st}, we know that $0\neq \mc\cap(\mg')^\bot=\ker (j:\mc\to\so(\mq))$. Hence, there exists an orthonormal basis $\{z,z',z''\}$ of $\mc$ such that $j(z)=0$ and $\tau(z')z''=z$.  For these elements, \eqref{eq:jcond} gives
\[
[j(z''),j(z)]=0=\frac23 \mu j(z'),\quad [j(z),j(z')]=0=\frac23 \mu j(z'').
\]Since $\mu\neq 0$, we get $j(z')=j(z'')=0$ as well and thus $\ker (j:\mc\to\so(\mq))=\mc=\mc\cap(\mg')^\bot$, which is a contradiction since $0\neq \mg'\subset \mc$. Whenever $\nabla^\lambda$ is an instanton, \eqref{it:jc} in Lemma \ref{lm:reduc2} holds, and thus $\mg'=\mc$.
\end{proof}

\begin{cor}\label{cor:comker}
    If $\nabla^\lambda$ is a $\rG_2$-instanton and $(\d e^i)^-\neq 0$ for some $i\in\{5,6,7\}$
    , then there exists a vector $0\neq x\in\mq$ such that $j(z)x=0$ for all $z\in\mg'$. 
\end{cor}

\begin{proof}
    Assume that $\nabla^\lambda$ is a $\rG_2$-instanton, then, by Corollary \ref{cor:g=c}, $\mg'=\mc$ and $\mr=\mq$. Moreover, \eqref{eq:57} and \eqref{eq:14} hold, so  \eqref{it:jc}  in Lemma \ref{lm:reduc2} implies that $\mh:=j(\mg')$ is a 3-dimensional Lie subalgebra of $\so(\mr)$. Clearly, 
     \begin{equation*}
        \mh=\{j(z)=j(z)^++j(z)^-\in \so(\mr)^+\oplus\so(\mr)^-:z\in\mg'\}.
    \end{equation*}
         Moreover, by    \eqref{it:Sid} in Lemma \ref{lm:reduc2} and  \eqref{eq:Sdiag2}, $j^+:\mg'\to\so(\mr)^+$ is also injective.        
         Identifying $\so(\mr)^\pm$ with the imaginary quaternions $\im\H$, we can define a map $f:\im\H\to\im\H$ where for $p\in\im\H\simeq \so(\mr)^+$, $f(p)=j^-((j^+)^{-1}(p))$. 
         
        The hypothesis $(\d e^i)^-=\alpha_i^-\neq 0$ for some $i$ implies that the map $j^-:\mg'\to\so(\mr)$ is non-trivial, so $f\neq 0$. In addition, $j^\pm:\mg'\to\so(\mr)^\pm$ satisfy \eqref{eq:jj+} and \eqref{eq:jcondm}, so $f$ is a Lie algebra homomorphism.

         Since $\im\H\simeq \sp(1)$ is simple, $f$ is actually an automorphism, and there  is some $a\in \Sp(1)$ such that $f(p)=apa^{-1}$, for all $p\in\im\H$. 
         Hence, 
\begin{equation}\label{eq:jzh}
\mh=\{p+apa^{-1}:p\in\im\H\}.
\end{equation}

Recall that for $p,q\in\im\H$, the element $p+q\in\so(\mr)=\H$ acts on $x\in\H\simeq \mq$ by the formula
\begin{equation}\label{eq:quatprod}
    (p+q)\cdot x=px-xq.
\end{equation}

If ${\bf 1}$ denotes the unit in $\H$, one can easily check that the element $x:={\bf 1}\cdot a^{-1}\in\H\simeq \mr$ verifies $(p+apa^{-1})\cdot x=0$ for all $p\in\im\H$, for the action in \eqref{eq:quatprod}. Since $\mh$ is as in \eqref{eq:jzh}, this shows our claim.
\end{proof}


\begin{lm}\label{lm:inst} 
The connection $\nabla^\lambda$ is a $\rG_2$ instanton if and only if $\dim\mg'=3$ (i.e. $\mc=\mg'$) and the following conditions hold:
\begin{enumerate}
\item \label{it:lambis} $\lambda=1$,
\item \label{it:deimu} $(\dd e^{i+4})^+=\frac{\mu}3\sigma_i^+$ for $i=1,2,3$, for some $\mu\neq 0$,
\item \label{it:jg} the map $j:\mg'\to\so(\mr)$ satisfies
\begin{equation}\label{eq:jcondbis}
[j(z),j(z')]=\frac23 \mu j(\tau(z)z'), \qquad \forall z,z'\in\mg',
\end{equation}where $\tau\in\so(\mg')$ is defined by in \eqref{eq:taudef}.
\end{enumerate}
\end{lm}

\begin{proof}
Notice first that \eqref{it:Sid} in Lemma \ref{lm:reduc2} is equivalent to \eqref{it:deimu} in the statement, by definition \eqref{eq:cc} of $S$  and \eqref{eq:Sdiag2}. In addition, if this condition holds, by \eqref{eq:brbeta} and \eqref{eq:j0}, we get
\begin{equation}\label{eq:finaleq}
\ad_u^+=-\frac{\mu}3\ad_u^0, \qquad j(z)^+=-\frac{\mu}3j(z)^0, \qquad \forall u\in\mr, z\in\mc.
\end{equation}Therefore, using \eqref{eq:npend1}-\eqref{eq:npend2}, we obtain that for any $u\in\mr$, $z\in\mc$,
\begin{eqnarray}\label{eq:nabesp1}
\nabla_u^\lambda&=&\frac{\mu(\lambda-1)}{6}(\ad_u^0-(\ad_u^0)^*)-\frac{(\lambda-1)}2(\ad_u^--(\ad_u^-)^*),\\
\nabla_z^\lambda&=&\frac{(\lambda+1)\mu}{6}j(z)^0-\frac{\lambda+1}2j(z)^--\frac23\mu\lambda\tau(z).\label{eq:nabesp2}
\end{eqnarray}

Suppose first that $\nabla^\lambda$ is an instanton so that, in particular,  \eqref{eq:57}--\eqref{eq:14} hold and thus (1)-(2)-(3) in Lemma  \ref{lm:reduc2} are satisfied. Then, clearly, (2) and (3) in the statement hold.
We will show that \eqref{eq:1457} does not hold for $\lambda=-1/3$, which will imply that (1) in the statement holds as well.

In fact, for $\lambda=-1/3$, \eqref{eq:nabesp1} and \eqref{eq:nabesp2} read
\begin{eqnarray*}
\nabla_u^{-1/3}&=&-\frac29\mu(\ad_u^0-(\ad_u^0)^*)\\
\nabla_z^{-1/3}&=&\frac{\mu}{9}j(z)^0+\frac29\mu\tau(z).
\end{eqnarray*}
For any $u,w\in\mr$, $z,z'\in\mc$, straightforward computations give
\begin{equation}\label{eq:Ruz13}
\lela R^{-1/3}_{u,z}w,z'\rira
=-\frac2{81}\mu^2\left(\lela j(z)^0w,j(z')^0u\rira+2\lela j(\tau(z)z')^0u,w\rira\right)
\end{equation}
Hence, in view of \eqref{eq:Ruv}, \eqref{eq:Rzz} and \eqref{eq:Ruz13}, one has
\begin{eqnarray*}
(R^{-1/3})^5_1&=&\sum_{\begin{smallmatrix}
j=1,
\ldots,4,\\
k=5, 6, 7\end{smallmatrix}}\lela R^{-1/3}_{e_j,e_k}e_1,e_5\rira e^j\wedge e^k\\
&=&-\frac2{81}\mu^2
\left( e^1\wedge e^5+  e^2\wedge e^6 +e^4\wedge e^7
+2\left(e^2\wedge e^6+ e^4\wedge e^7\right)\right) .
\end{eqnarray*}

Since $\psi$ has the form in \eqref{eq:stphi}, we finally obtain
\begin{equation}
(R^{-1/3})^5_1\wedge \psi
= -\frac{10}{81}\mu^2
e^{124567},
\end{equation}which is non-zero because $\mu\neq 0$. 
Therefore, \eqref{eq:1457} does not hold for $\lambda=-\frac13$.

For the converse, assume (1)--(3) in the statement hold. By Lemma \ref{lm:reduc2}, \eqref{eq:57} and \eqref{eq:14} hold so we only need to check \eqref{eq:1457}.

When $\lambda=1$,    $\nabla_u^1=0$ by \eqref{eq:nabesp1}, and thus $R^1_{u,y}=0$ for all $u\in\mr$, $y\in\mg$. This, together with \eqref{eq:Rlamab}, implies $(R^1)_\beta^\alpha\in \Lambda^2\mc^*$ for every $\alpha\in \{1, \ldots,4\}$, $\beta\in\{5, 6,7\}$.
However, using \eqref{eq:Rzz}, one can easily check that for any $z,z',z''\in\mc$ and $u\in\mr$, $\lela R^1_{z,z'}u,z''\rira=0$ when $\lambda =1$ and \eqref{eq:finaleq} holds. Therefore,  $(R^1)_\beta^\alpha=0$ for every $\alpha=1, \ldots, 4$ and $\beta=5,\ldots, 7$ which implies \eqref{eq:1457}. Consequently, $\nabla^1$ is a $\rG_2$-instanton.
\end{proof}

\bibliographystyle{plain}
\bibliography{biblio}

\end{document}